\documentclass[a4paper, 11pt]{amsart}
\usepackage{amsmath,amsthm,amssymb,amsfonts,color}
\usepackage[arrow,matrix]{xy}

\theoremstyle{plain}
\newtheorem{theorem}{Theorem}[section]
\newtheorem{proposition}[theorem]{Proposition}
\newtheorem{lemma}[theorem]{Lemma}
\newtheorem{corollary}[theorem]{Corollary}
\numberwithin{equation}{section}

\theoremstyle{definition}
\newtheorem{definition}[theorem]{Definition}
\newtheorem{remark}[theorem]{Remark}
\newtheorem{example}[theorem]{Example}

\newtheorem{problem}[theorem]{Problem}

\newcommand{\C}{\mathbb{C}}
\newcommand{\uC}{\underline{\mathbb{C}}}
\newcommand{\Q}{\mathbb{Q}}
\newcommand{\R}{\mathbb{R}}
\newcommand{\Z}{\mathbb{Z}}

\newcommand{\CP}{\mathbb{C}P}
\newcommand{\cF}{\mathcal{F}}
\newcommand{\blambda}{\boldsymbol{\lambda}}

\newcommand{\bpsi}{\bar{\psi}}

\newcommand{\CC}{\underline{\mathbb{C}}}

\newcommand{\n}{n}

\newcommand{\atopp}[2]{\genfrac{}{}{0pt}{}{#1}{#2}}

\DeclareMathOperator{\Hom}{Hom}

\newcommand{\toricmanifold}{non-singular complete toric variety }
\newcommand{\toricmanifolds}{non-singular complete toric varieties }
\newcommand{\smoothcompact}{non-singular complete }
\newcommand{\projectivetoricmanifold}{non-singular projective toric variety }

\begin{document}
\title[Projective bundles over toric surfaces]{Projective bundles over toric surfaces}

\author[S.Choi]{Suyoung Choi}
\address{Department of Mathematics, Ajou University, San 5, Woncheondong, Yeongtonggu, Suwon 443-749, Korea}
\email{schoi@ajou.ac.kr}
%

\author[S.Park]{Seonjeong Park}
\address{Department of Mathematical Sciences, KAIST, 291 Daehak-ro, Yuseong-gu, Daejeon 305-701, Korea}
\email{psjeong@kaist.ac.kr}

\thanks{The first author is partially supported by Basic Science Research Program through the National Research Foundation of Korea(NRF) funded by the Ministry of Education, Science and Technology(2012-0008543), and is supported by TJ Park Science Fellowship funded by POSCO TJ Park Foundation. The second author is supported by the second stage of the Brain Korea 21 Project, the Development of Project of Human Resources in Mathematics, KAIST in 2012.}


\date{\today}
\begin{abstract}
    Let $E$ be the Whitney sum of complex line bundles over a topological space $X$. Then, the projectivization $P(E)$ of $E$ is called a \emph{projective bundle} over $X$. If $X$ is a non-singular complete toric variety,
    so is $P(E)$. In this paper, we show that the cohomology ring of a non-singular projective toric variety $M$ determines whether it admits a projective bundle structure over a \smoothcompact toric surface. In addition, we show that two $6$-dimensional projective bundles over
    $4$-dimensional quasitoric manifolds are diffeomorphic if their cohomology rings are isomorphic as graded rings. Furthermore, we study the smooth classification of higher dimensional projective bundles over $4$-dimensional quasitoric manifolds.
\end{abstract}
\maketitle

\section{Introduction}\label{sec:introduction}
    A \emph{toric variety} is a normal algebraic variety of complex dimension $n$ with an action of the algebraic torus $(\C^\ast)^n$ having an open dense orbit.
    A typical example of a \toricmanifold is the complex projective space $\CP^n$ with a linear action of $(\C^\ast)^n$. A toric variety of complex dimension $2$ is called a \emph{toric surface}. It is well known that every \smoothcompact toric surface is projective and that it can be obtained by blow-ups from either $\CP^2$ or one of the Hirzebruch surfaces as an algebraic variety.

    Let $X$ be a topological space and
    $L_i \to X$ a complex line bundle over $X$ for $i=0, \ldots,\n$. For the Whitney sum $E=\bigoplus_{i=0}^n L_i$, we define the projectivization $P(E)$ of $E$ by taking the projectivization of each fiber of $E$. Then, $P(E)$ is a fiber bundle over
    $X$ with the fiber space $\CP^n$. The space $P(E)$ is called a \emph{projective bundle} over $X$.

    Assume that $X$ is a \projectivetoricmanifold of complex dimension $k$. Note that each $L_i$ has a $\C^\ast$-action as a scalar multiplication, and hence, the Whitney sum $E=\bigoplus_{i=0}^n L_i$ has a $(\C^\ast)^{k+n+1}$-action. Hence, the projectivization $P(E)$ of $E$ has an induced $(\C^\ast)^{k+\n}$-action and is also a non-singular projective toric variety.
    From this viewpoint, we can construct an interesting subclass of toric varieties as follows. Starting with $X$ as a point and repeating the above construction $h$ times ($h \in \Z_{\geq0}$), we obtain a new non-singular projective toric variety, which is called an \emph{$h$-stage generalized Bott manifold} (see \cite{ch-ma-su08} for details). In particular, when the complex dimension of a fiber is equal to $1$ at each fibration, the total space is simply referred to as a \emph{Bott manifold}.

    We observe that $\CP^2$ and any two-stage Bott manifold are \smoothcompact toric surfaces. Hence, two-stage generalized Bott manifolds over $\CP^2$ and three-stage Bott manifolds can be regarded as projective bundles over \smoothcompact toric surfaces.

    In this study, we investigate a projective bundle over a \smoothcompact toric surface $S$. Especially, we focus on how much topological information is contained in the cohomology ring of toric varieties.

    Firstly, we find the necessary and sufficient condition in terms of a cohomology ring for a \projectivetoricmanifold $X$ to admit a projective bundle structure over $S$.

    \begin{theorem}\label{main1}
        A \projectivetoricmanifold is equivalent to a projective bundle over a \smoothcompact toric surface if and only if its integral cohomology ring is isomorphic to that of some projective bundle over a \smoothcompact toric surface as graded rings.
    \end{theorem}

    Secondly, we are also interested in classifying such manifolds topologically or smoothly. So far, many results on the smooth classification of generalized Bott manifolds have been established. We refer the reader to a survey paper~\cite{ch-ma-su11} on this topic. Remarkably, the results lead us to conjecture that all \toricmanifolds are smoothly classified by their cohomology rings. This problem is now called the \emph{cohomological rigidity problem} for toric varieties.

    On the other hand, for a \projectivetoricmanifold $X$, the algebraic torus action of $(\C^\ast)^n$ induces a locally standard action of the compact subtorus $T^n=(S^1)^n$. Moreover, the orbit space $X/T^n$ can be identified with a simple polytope
    The identification means that there is an orbit map $\rho\colon X\to P$ that maps every $k$-dimensional orbit to a point in the interior of a $k$-dimensional face of $P$ for $k=0,\ldots,n$.

    The topological analogue of a non-singular projective toric variety, called a \emph{quasitoric manifold}, was introduced by Davis and Januszkiewicz \cite{DJ}\footnote{The authors would like to indicate that the notion of quasitoric manifolds originally appeared under the name ``toric manifolds'' in \cite{DJ}. Later, it was renamed in \cite{BP} in order to avoid confusion with \smoothcompact toric varieties. As far as the authors know, there has been a dispute about the terminology. The authors have no preference; however, in this paper, they follow the terminology used in their previous papers.}.
    A $2n$-dimensional closed smooth manifold $M$ is called a quasitoric manifold if it has a locally standard action of an $n$-dimensional compact torus $T^n$, and the orbit space $M/T^n$ can be identified with a simple polytope of dimension $n$.
    Unlike non-singular projective toric varieties, quasitoric manifolds do not necessarily admit almost complex structures. For example, $\CP^2 \# \CP^2$ is not a non-singular complete toric variety, while it is a quasitoric manifold with an appropriate $T^2$-action. Hence, the class of quasitoric manifolds is larger than that of non-singular projective toric varieties.

    We classify projective bundles over $4$-dimensional quasitoric manifolds up to diffeomorphism. Since a \smoothcompact toric surface is a $4$-dimensional quasitoric manifold, our results give a partial affirmative answer to the cohomological rigidity problem for toric varieties. Throughout this paper, $H^\ast(X)$ denotes the integral cohomology ring of a topological space $X$, and $\beta_2 (X)$ denotes the second Betti number of $X$, that is, the rank of $H^2(X)$ over $\Z$.

    \begin{theorem}\label{main2}
        Let $M$ and $M'$ be projective bundles over $4$-dimensional quasitoric manifolds with the fiber space $\CP^1$, respectively. If $H^\ast(M)\cong H^\ast(M')$ as graded rings, then $M$ and $M'$ are diffeomorphic.
    \end{theorem}

    Let $\mathcal{B}$ be the set of $4$-dimensional quasitoric manifolds which cannot be expressed as $\CP^2\#n\overline{\CP^2}$ or $n\CP^2\#\overline{\CP^2}$ for $n>9$.

    \begin{theorem}\label{main3}
        Let $B,B' \in \mathcal{B}$, and let $M$ and $M'$ be projective bundles over $B$ and $B'$, respectively. Then, any graded ring isomorphism from $H^\ast(M)$ to $H^\ast(M')$ is induced by a diffeomorphism from $M'$ to $M$.
    \end{theorem}

    The remainder of this paper is organized as follows. In Section~\ref{sec:projective bundles over a smoothcompact toric surface}, we investigate a projective bundle over a \smoothcompact toric surface. In Section~\ref{sec:quasitoric manifolds}, we recall the definitions and properties of quasitoric manifolds, and we discuss a projective bundle as a quasitoric manifold. In Section~\ref{sec:polytope rigidity}, we show that  the product of a simplex and a polygon is a combinatorially rigid polytope (the definition is given below). This implies that if a quasitoric manifold has the cohomology ring isomorphic to that of a projective bundle over a $4$-dimensional quasitoric manifold, then their orbit spaces are combinatorially equivalent.
    In Section~\ref{sec:cohomology determines}, we prove Theorem~\ref{main1}.
    In the last two sections, we prove Theorems~\ref{main2} and \ref{main3}.

\section{Cohomology of projective bundles over toric surfaces}\label{sec:projective bundles over a smoothcompact toric surface}

    Let $B$ be a smooth manifold, and let $E$ be a complex vector bundle over $B$ with the fiber space $V$. We define the projectivization $P(E)$ of $E$ by taking the projectivization of each fiber of $E$. Then, $P(E)$ is a fiber bundle over $B$ with the fiber space $P(V)$.

    Let $x$ be the negative of the first Chern class of the tautological line bundle over $P(E)$. Then, $H^\ast(P(E))$ can be regarded as an algebra over $H^\ast(B)$ via $\pi^\ast\colon H^\ast(B)\to H^\ast(P(E))$, where $\pi\colon P(E)\to B$ denotes the projection. When $H^\ast(B)$ is finitely generated and torsion free, $\pi^\ast$ is injective, and $H^\ast(P(E))$ as an algebra over $H^\ast(B)$ is known to be described as
    \begin{equation*}
        H^\ast(P(E))=H^\ast(B)[x] \left/\left(\sum_{k=0}^n c_k(E)x^{n-k}\right)\right.,
    \end{equation*}
    where $n$ is the complex dimension of $V$ and $c_k (E)$ is the $k$-th Chern class of $E$ (see~\cite{BH}).

    \begin{lemma}\cite{ch-ma-su08}\label{lem:isomorphic projective bundles}
        Let $B$ and $E$ be as above, and let $L$ be a complex line bundle over $B$. Let $E^\ast$ denote the complex vector bundle dual to $E$. Then, $P(E\otimes L)$, $P(E)$, and $P(E^\ast)$ are isomorphic as bundles over $B$; in particular, they are diffeomorphic.
    \end{lemma}

    Let $S$ be a \smoothcompact toric surface. Then, $S$ is obtained by blow-ups from  either $\CP^2$ or one of the Hirzebruch surfaces (see~\cite{F}). In particular, $S$ is diffeomorphic to $\CP^2$, $\CP^1\times\CP^1$, or the connected sum of $\CP^2$ with a finite number of copies of $\overline{\CP^2}$, where $\overline{\CP^2}$ denotes $\CP^2$ with reversed orientation (see~\cite{FY}). Consequently, $H^\ast(S)$ is finitely generated as a ring by the second cohomology classes, and it is torsion free. Hence,
    we may assume that $H^\ast(S)$ is generated by $x_1,\ldots,x_m$ of degree $2$, where $\beta_2(S) = m$.

    Now, let $E$ be the Whitney sum of $n+1$ complex line bundles over $S$. Then, $P(E)$ is a projective bundle over $S$ with the fiber space $\CP^n$. By Lemma~\ref{lem:isomorphic projective bundles}, we may assume that $$P(E)=P(\uC\oplus L_1\oplus\cdots\oplus L_n),$$ where $\uC$ is the trivial complex line bundle and $L_i$'s are complex line bundles over $S$. Then, $H^\ast(P(E))$ is a free module over $H^\ast(S)$ with basis $\{1,x,\ldots,x^n\}$, and the ring structure is determined by the single relation
    $$x^{n+1}+c_1(E)x^n+\cdots+c_n(E)x=0,$$
    where $x$ is the negative of the first Chern class of the tautological line bundle over $P(E)$ and $c(E)=\sum_{k \geq 0} c_k (E)=\prod_{i=1}^nc(L_i)$. Note that the first Chern class is a complete invariant for classifying complex line bundles smoothly.
    Put $c_1(L_i):=\sum_{j=1}^m a_{ij}x_j$ for each $i=1,\ldots,n$. Then $$c(E)=\prod_{i=1}^n\left(1+\sum_{j=1}^ma_{ij}x_j\right).$$
    Therefore, the cohomology ring of $P(E)$ is written as follows:
    \begin{equation*}
        H^\ast(P(E)) = H^\ast(S)[x] \left/ x \prod_{i=1}^n \left(x+ \sum_{j=1}^m a_{ij}x_j \right) \right..
    \end{equation*}

    \begin{remark}
        If a complex vector bundle $E$ over a \toricmanifold is not isomorphic to a Whitney sum of complex line bundles, then the projectivization $P(E)$ is not necessary to be a toric variety.
    \end{remark}

    We call $P(E)$ a \emph{projective bundle over} a smooth manifold $B$ only when $E$ is the Whitney sum of complex line bundles over $B$. In particular, if $P(E)$ is a projective bundle over a \smoothcompact toric surface $S$, then it is also a non-singular projective toric variety, as described in the introduction.

    \begin{example}
        Let $E_1:=\uC\oplus\bigoplus_{i=1}^n L_i$ and $E_2:=\uC\oplus\bigoplus_{i=1}^n L_i'$ be complex vector bundles over $S:=\CP^2$, where $L_i$'s and $L_i'$'s are complex line bundles over $S$. It is shown in~\cite{ch-ma-su10} that
        $$P(E_1)\approx P(E_2)\mbox{ if and only if } H^\ast(P(E_1))\cong H^\ast(P(E_2)),$$ where $\approx$ denotes ``diffeomorphic'' and $\cong$ denotes ``isomorphic as rings''.
    \end{example}
    This example shows that the cohomology ring determines the smooth type of a projective bundle $P(E)$ over $\CP^2$. Hence, we may ask the following question.

    \begin{problem}\label{prob:rigidity for projective bundles}
        Are projective bundles $P(E_1)$ and $P(E_2)$ over a \smoothcompact toric surface $S$ diffeomorphic or homeomorphic if $H^\ast(P(E_1))\cong H^\ast(P(E_2))$ as rings?
    \end{problem}

    We investigate this problem further in Sections~\ref{sec:cohomological rigidity and strong cohomological rigidity} and \ref{section:higher_projective_bdl}, where partial affirmative answers are given.

\section{Quasitoric manifolds}\label{sec:quasitoric manifolds}

    Since a projective bundle $P(E)$ over a \smoothcompact toric surface is a non-singular projective toric variety, it is also a quasitoric manifold. In this section, we recall general facts about quasitoric manifolds, and we concern $P(E)$ from a different point of view, that is, as a quasitoric manifold.

    Let $M$ be a $2n$-dimensional quasitoric manifold with an orbit map $\rho\colon M \to P$. Then,
    for a codimension-$k$ face $F$ of $P$, the preimage $\rho^{-1}(F)$ is a connected codimension-$2k$ submanifold of $M$, which is fixed pointwise by a $k$-dimensional subgroup of $T^n$. Let $\mathcal{F}(P) = \{F_1, \cdots , F_d\}$ be the set of facets, codimension-one faces, of $P$. Then, there is a primitive vector $\lambda_i$ in the integer lattice $\Z^n = \Hom(S^1, T^n)$ of one-parameter circle subgroups in $T^n$ such that $\lambda_i$ spans the circle subgroup $T_{F_i}\subset T^n$ fixing the \emph{characteristic submanifold} $\rho^{-1}(F_i)$. Hence, the vector $\lambda_i$ is determined up to sign.
    Then, the function $\boldsymbol{\lambda}\colon F_i \mapsto \lambda_i$ is called the \emph{characteristic function} of $M$, and it
    satisfies the \emph{non-singularity condition}:
    \begin{equation}\label{non-singularity}
    \begin{array}{l}
    \blambda(F_{i_1}), \ldots, \blambda(F_{i_{\alpha}}) \mbox{ form a part of an integral basis of } \Z^n\\
    \mbox{whenever the intersection } F_{i_1} \cap \cdots \cap F_{i_{\alpha}} \mbox{ is non-empty.}
    \end{array}
    \end{equation}

    Let $P$ be a simple polytope of dimension $n$ and let $\mathcal{F}(P)$ be the set of facets of $P$. For a function $\blambda \colon \mathcal{F}(P) \to \Z^n$ satisfying the non-singularity condition~\eqref{non-singularity}, let $T_F$ denote the subgroup of $T^n$ represented by the unimodular subspace of $\Z^n$ spanned by $\blambda(F_{i_1}),\ldots,\blambda(F_{i_\alpha})$, where $F=F_{i_1}\cap\cdots\cap F_{i_\alpha}$. Given a pair $(P,\blambda)$, we can construct a manifold
    \begin{equation}\label{eqn:construction of a quasitoric manifold}
        M(P,\blambda) := T^n \times P /\sim,~(t,p) \sim (s,q)\Leftrightarrow \mbox{$p=q$ and $t^{-1}s \in T_{F(p)}$},
    \end{equation}
    where $F(p)$ is the face of $P$ that contains a point $p\in P$ in its relative interior. Then, the standard $T^n$-action on $T^n$ descends to a locally standard action of $T^n$ on $M(P,\blambda)$ whose orbit space is combinatorially equivalent to $P$. Hence, $M(P,\blambda)$ is indeed a quasitoric manifold with the characteristic function $\blambda$.

    It is shown in \cite{DJ} that for a quasitoric manifold $M$ over $P$ with its characteristic function $\blambda$, there is an equivariant homeomorphism $M\to M(P,\blambda)$ covering the identity on $P$. Thus, any quasitoric manifold is expressed by a pair of a simple polytope $P$ and a function $\blambda\colon\mathcal{F}(P)\to\Z^n$ satisfying the non-singularity condition \eqref{non-singularity}.

    Note that one may assign an $n \times d$ matrix $\Lambda$, called a \emph{characteristic matrix}, to a characteristic function $\blambda$ by $$ \Lambda = ( \blambda(F_1) \cdots \blambda(F_d) ) = (A|B),$$ where $A$ is an $n \times n$ matrix and $B$ is an $n \times (d-n)$ matrix.
    From the non-singularity condition \eqref{non-singularity}, if we assume that $F_1\cap\cdots\cap F_n\neq\emptyset$, $A$ is invertible.
    Hence, in this paper, for simplicity, we assume that the first $n$ columns of $\Lambda$ form an invertible matrix $A$, and we sometimes denote $M(P, \blambda)$ by $M(P, \Lambda)$ as long as there is no confusion.

    We say that two quasitoric manifolds $M_1$ and $M_2$ over the same polytope $P$ are \emph{equivalent} if there is a $\theta$-equivariant homeomorphism $f\colon M_1 \rightarrow M_2$ that covers the identity on $P$, where $\theta$ is an automorphism on $T^n$.
    Here, ``$\theta$-equivariant homeomorphism $f$" means that the homeomorphism $f$ satisfies $f(t\cdot x)=\theta(t)\cdot f(x)$ for all $t\in T^n$ and $x \in M$.

    Assume that $M_1 = M(P, \blambda_1)$ and $M_2 = M(P, \blambda_2)$. If there is a general linear group $\sigma \in GL_n(\Z)$ such that $\blambda_1 = \sigma \circ \blambda_2$, then $M_1$ and $M_2$ are equivalent. Hence, for each quasitoric manifold $M(P, \blambda)$, the corresponding matrix $\Lambda$ can be represented by $(E_n | A^{-1}B)$, where $E_n$ is the identity matrix of order $n$.

     Let $P_i$ be an $n_i$-dimensional simple polytope with $\mathcal{F}(P_i)=\{F_{i,1},\ldots,F_{i,d_i}\}$ for $i=1,2$. Then, the Cartesian product of two simple polytopes $P=P_1\times P_2$ is an $(n_1+n_2)$-dimensional simple polytope with $(d_1+d_2)$ facets.
     Note that each facet of $P$ is of the form either $F_{1,j}\times P_2$ or $P_1\times F_{2,j}$. For convenience, we shall give an order on $\cF(P)$ by
     $$
        \cF(P)=\{F^1_1, \ldots, F^1_{n_1}, F^2_{1}, \ldots, F^2_{n_2}, F^1_{n_1+1},\ldots,F^1_{d_1}, F^2_{n_2+1},\ldots, F^2_{d_2}\},
     $$ where $F^1_{j}=F_{1,j}\times P_2$ and $F^2_{j}=P_1\times F_{2,j}$. Now, let $M$ be a quasitoric manifold over the polytope $P$, and set $n=n_1 + n_2$. Then, we obtain a characteristic function $\blambda : \mathcal{F}(P) \rightarrow \Z^n$. Up to equivalence, we may assume that the characteristic matrix $\Lambda$ associated with $\blambda$ is of the form
    \begin{equation}\label{char}
        \Lambda=\left(
            \begin{array}{cccc}
                E_{n_1} &    O     & A_{11} & A_{12} \\
                O& E_{n_2} &A_{21} & A_{22} \\
            \end{array}
            \right),
    \end{equation}
    where $A_{ij}$ is an $n_i \times (d_j -n_j)$ matrix.
    \begin{lemma}\label{block_char}
        A sub-matrix $\Lambda_i := (E_{n_i}, A_{ii})$ of $\Lambda$ is a characteristic matrix on $P_i$ for $i=1,2$.
    \end{lemma}
    \begin{proof}
        Note that for any vertex $v = F_{1,i_1} \cap \cdots \cap F_{1, i_{n_1}}$ of $P_1$, $$v\times(F_{2,1}\cap \cdots \cap F_{2,n_2}) = F^1_{i_1} \cap \cdots \cap F^1_{i_{n_1}} \cap F^2_1 \cap \cdots \cap F^2_{n_2}$$ is also a vertex of $P$. Define $\blambda_1 \colon \cF(P_1) \to \Z^{n_1}$ by mapping $\blambda_1(F_{1,j})$ to the $j$-th column vector of $\Lambda_1$. Then,
        \begin{equation*}
            \begin{split}
                &\det(\blambda_1(F_{1,i_1}) , \cdots ,\blambda_1(F_{1,i_n}))\\
                &\quad = \det(\blambda(F^1_{i_1}) ,\cdots, \blambda(F^1_{i_n}), \blambda(F^2_1), \cdots, \blambda(F^2_{n_2}))\\
                &\quad  = \pm 1.
            \end{split}
        \end{equation*}
         Therefore, $\blambda_1$ satisfies the non-singularity condition on $P_1$, and hence, $\Lambda_1$ is a characteristic matrix on $P_1$. A similar argument shows that $\Lambda_2$ is also a characteristic matrix on $P_2$.
    \end{proof}
    \begin{definition}\label{def:equiv. bundle}
        Let $M$, $F$, and $B$ be quasitoric manifolds of dimensions $2n_1+2n_2$, $2n_1$, and $2n_2$, respectively. A bundle $\pi:M \rightarrow B$ with fiber $F$ is said to be \emph{equivariant} if there is a surjective homomorphism $\theta:T^{n_1+n_2} \rightarrow T^{n_2}$ such that $\pi(t\cdot x)=\theta(t)\cdot \pi(x)$ for all $t\in T^{n_1+n_2}$ and $x \in M$, the fiber $\pi^{-1}(b)$ has a locally standard action of $\ker \theta$ for each $b \in B$, and $\pi^{-1}(b)$ is equivalent to $F$.
    \end{definition}
    If $M\to B$ is an equivariant bundle with fiber $F$, then the orbit space of $M$ is the product of the orbit space of $B$ and the orbit space of $F$ by Proposition~5 in~\cite{Do}. Furthermore, we have the following lemma, which is an immediate corollary of Theorem~6 in \cite{Do}.
    \begin{lemma} \label{lem:chr_eq_bdl}
        A quasitoric manifold $M$ over $P_1 \times P_2$ is an equivariant bundle with fiber $M(P_1, \Lambda_1)$ and base $M(P_2, \Lambda_2)$ if and only if it is equivalent to $M(P_1 \times P_2, \Lambda)$, where $\Lambda$ is the characteristic matrix of the form
        \begin{equation*}
            \left(
            \begin{array}{cccc}
                E_{n_1} & O       & A_{11} & A_{12} \\
                O       & E_{n_2} & O & A_{22}\\
            \end{array}
            \right)
        \end{equation*}
        and $\Lambda_i \; :=\; (E_{n_i},A_{ii})$ for $i=1,2$.
    \end{lemma}

    Let us consider a projective bundle $P(E)$ over a \smoothcompact toric surface $S$ with $\beta_2 (S) = m$. We remark that $S$ is a quasitoric manifold of (real) dimension $4$ with a natural $T^2$-action, and its orbit space $S/T^2$ is an $(m+2)$-gon $G(m+2)$. Assume that $P(E)$ is a $\CP^n$-bundle over $S$. Then, $P(E)$ is also a quasitoric manifold of dimension $2n+4$, and its orbit space is $\Delta^n \times G(m+2)$, where $\Delta^n$ denotes an $n$-simplex.

    Let us find the characteristic matrix of $P(E)=P(\CC\oplus\bigoplus_{i=1}^n L_i)$ explicitly.
    Put $\mathcal{F}(G(m+2))=\{F_1,\ldots,F_{m+2}\}$ so that $F_i\cap F_{i+1} \neq\emptyset$ for $i=1, \ldots, m+1$ and $F_{m+2} \cap F_{1} \neq \emptyset$.
    Let $\rho\colon S\to G(m+2)$ be the orbit map, and let the characteristic matrix\footnote{Since $S$ is projective, the columns of $\Lambda_S$ can generate the normal fan of a polygon in $\R^2$. That is to say that the polygon $\{\mathbf{x}\in\R^2\,|\,\langle\mathbf{x},\blambda_S(F_i)\rangle\geq 0 \mbox{ for all } i=1,\ldots,m+2\}$ is identified with the orbit space of $S$.} of $S$ be given by
    \begin{equation}\label{eqn:char. mx of toric surface}
            \Lambda_S=\left(
            \begin{array}{ccccc}
                1 & 0 & -\Lambda_1^1 & \cdots & -\Lambda_1^m \\
                0 & 1 & -\Lambda_2^1 & \cdots & -\Lambda_2^m \\
            \end{array}
            \right),
    \end{equation} where the order of facets is $F_{m+1},F_{m+2},F_1,\ldots,F_m$.
    Note that $S$ has a $T^2$-invariant complex structure. Thus, every normal bundle $\nu_j:=\nu(S_j\subset S)$ over a characteristic submanifold $S_j:=\rho^{-1}(F_j)$, $j=1,\ldots,m+2$, is a $T^2$-invariant complex line bundle.
    Let $\gamma_j$ be the $T^2$-invariant complex line bundle over $S$ extending the normal bundle $\nu_j$ of the characteristic submanifold $S_j$ trivially outside the tubular neighborhood of $S_j$. Then, the first Chern class $c_1(\gamma_j)\in H^2(S)$ is dual to the characteristic submanifold $S_j$ for $j=1,\ldots,m+2$. Since $c_1(\gamma_1),\ldots,c_1(\gamma_m)$ generate $H^2(S)$ and every complex line bundle is classified by its first Chern class, each complex line bundle $L_i$ over $S$ is isomorphic to $\gamma_1^{a_{i1}}\otimes \cdots \otimes\gamma_m^{a_{im}}$ for some integers $a_{i1},\ldots,a_{im}$. Hence, $P(E)$ is isomorphic to $P(\CC\oplus\bigoplus_{i=1}^n(\bigotimes_{j=1}^m\gamma_j^{a_{ij}}))$ for some integers $a_{ij}\in\Z$.

    \begin{proposition}\label{prop:char_ftn_projective_bundle}
        Let $S$ and $\gamma_j$ be given as above. Let $L_i=\bigotimes_{j=1}^m\gamma_j^{a_{ij}}$ for $i=1,\ldots,n$.
        Then, $M=P(\uC\oplus\bigoplus_{i=1}^n L_i)$ is a quasitoric manifold whose characteristic matrix is of the form
        \begin{equation}\label{eqn:char. matrix of proj. bundle over a toric mfd 1}
            \Lambda=\left(
            \begin{array}{ccc|cc|c|ccc}
                1 & \cdots & 0 & 0 & 0 & -1 & -a_{11} & \cdots & -a_{1,m} \\
                \vdots & \ddots & \vdots & \vdots & \vdots & \vdots & \vdots & \ddots & \vdots \\
                0 & \cdots & 1 & 0 & 0 & -1 & -a_{n,1} & \cdots & -a_{n,m} \\
                \hline
                0 & \cdots & 0 & 1 & 0 & 0 & -\Lambda_1^1 & \cdots & -\Lambda_1^m \\
                0 & \cdots & 0 & 0 & 1 & 0 & -\Lambda_2^1 & \cdots & -\Lambda_2^m \\
            \end{array}
            \right),
        \end{equation}
        where the order of facets of $\Delta^n \times G(m+2)$ is
        $$F_1^1,\ldots,F_n^1, F_{m+1}^2,F_{m+2}^2, F_{n+1}^1, F_1^2,\ldots,F_m^2.$$
    \end{proposition}
    \begin{proof}
        Let $X^G = \{x \in X \mid g\cdot x = x \text{ for all }g\in G\}$ for a topological space $X$ with a group action of $G$.
        Note that for a fixed point $p\in S^{T^2}$, we know that
        \begin{itemize}
            \item $\left.\gamma_i\right|_p=\left.\nu_i\right|_p$ as $T^2$-modules if $p\in S_i^{T^2}$, and
            \item  $\left.\gamma_i\right|_p$ is the trivial $T^2$-module if $p \not\in S_i^{T^2}$.
         \end{itemize}

         Remark that $\blambda_S(F_j)(S^1) \subset T^2$ is the circle subgroup that fixes $S_j \subset S$.
         Since the $T^2$-action on $S$ is effective, the action of $\blambda_S(F_j)(s)$ (for $s\in S^1$) on $\left.\gamma_i\right|_{S_j}$ is the complex multiplication on fibers by $s\in S^1\subset\C$ when $i=j$, and trivial if  $i \neq j$.

        On the other hand, we can give an additional $T^n$-action on a $T^2$-invariant complex vector bundle $\CC\oplus\bigoplus_{i=1}^n L_i$ by
        \begin{equation}\label{eqn:action on fibers}
            (t_1,\ldots,t_n)\cdot(u_0,u_1,\ldots,u_n)=(u_0,t_1u_1,\ldots,t_nu_n).
        \end{equation}
        Then, the total space of the projective bundle $\pi\colon M=P(\CC\oplus\bigoplus_{i=1}^n L_i)\to S$ has a $T^{n+2}$-action and is a quasitoric manifold with characteristic submanifolds $M_j:=\pi^{-1}(S_j)$ for $j=1,\ldots,m+2$ and $N_i:=\{u_i=0\}$ for $i=0,1,\ldots,n$.

        To find the characteristic matrix of $M$, we need to know which circle subgroup of $T^{n+2}$ fixes each characteristic submanifold. The $i$-th component of $T^{n+2}$ fixes $N_i$ when $i=1,\ldots,n$, and the circle subgroup $$\{(t^{-1},\ldots,t^{-1},1,1)\in T^{n+2}\,|\,t\in S^1\} \subset T^{n+2}$$ fixes $N_0$. The action of $\blambda_S(F_j)(s)$ ($s\in S^1$) on fibers of $\left.L_i\right|_{S_j}$ is the complex multiplication by $s^{a_{ij}}$ when $1\leq j\leq m$, and is trivial when $j=m+1,m+2$, while the action of $(t_1,\ldots,t_n)$ on $L_i$ is the complex multiplication by $t_i$ by~\eqref{eqn:action on fibers}. Therefore, in order to make the fiberwise action trivial, we have
        $$(t_1,\ldots,t_n)=(s^{-a_{1j}},\ldots,s^{-a_{nj}})$$
        when $1\leq j\leq m$ and $$(t_1,\ldots,t_n)=(1,\ldots,1)$$ when $j$ is $m+1$ or $m+2$. Thus, the proposition is proved.
    \end{proof}

    If $B$ is a $4$-dimensional quasitoric manifold that is not a \smoothcompact toric surface, then it does not necessarily admit a $T^2$-invariant almost complex structure. However, in this case, we can give an additional structure on $B$, the \emph{omniorientation}, which is defined to be a choice of an orientation for $B$ and of orientations for each of characteristic submanifolds $B_i$, $i=1,\ldots,m+2$. Then, the omniorientation determines an orientation for every normal bundle $\nu_i$. Since every $\nu_i$ is a two-plane bundle, it follows that an orientation of $\nu_i$ enables us to interpret $\nu_i$ as a complex line bundle. Since the torus $T^2$ is oriented, choosing an orientation for $G(m+2)$ is equivalent to choosing an orientation for $B$. Since each circle subgroup $T_{F_i}$ fixing $S_i$ acts on the normal bundle $\nu_i$, a choice of an omniorientation for $B$ is equivalent to a choice of an orientation for $G(m+2)$ together with an unambiguous choice of column vectors of $\Lambda_B$ (see Chapter~5 in \cite{BP} for details).

    From a similar argument to the proof of Proposition~\ref{prop:char_ftn_projective_bundle}, we obtain the following corollary.

    \begin{corollary}
        If $B$ is a $4$-dimensional quasitoric manifold with an omniorientation, then there is a projective bundle over $B$ whose characteristic matrix is of the form~\eqref{eqn:char. matrix of proj. bundle over a toric mfd 1}.
    \end{corollary}

    We close this section with a brief review of the cohomology ring of a quasitoric manifold.
    Let $P$ be an $n$-dimensional simple polytope with $d$ facets and $\cF(P) = \{F_1, \ldots, F_d\}$.
    Let $\mathbf{k}[v_1, \ldots, v_d]$ denote the polynomial ring in $d$ variables over a commutative ring $\mathbf{k}$ with unit, $\deg v_i = 2$. We primarily assume that $\mathbf{k}$ is a ring of integers $\Z$ or a ring of rational numbers $\Q$.
    We identify each $F_i \in \cF(P)$ with the indeterminate $v_i$. The \emph{face ring} (or \emph{Stanley-Reisner ring}) $\mathbf{k}(P)$ of $P$ is the quotient ring
    $$\mathbf{k}(P) =\mathbf{k}[v_1, \ldots, v_d]/I_P,$$
    where $I_P$ is the ideal generated by the monomials $v_{i_1}\cdots v_{i_\ell}$ whenever $F_{i_1}\cap\cdots\cap F_{i_\ell}=\emptyset$. The ideal $I_P$ is called the \emph{Stanley-Reisner ideal} of $P$.

    Let $M$ be a quasitoric manifold over $P$ with the projection $\rho\colon M \to P$ and the characteristic function $\blambda$. Then, one can find an isomorphism between the face ring $\Z(P)$ and the equivariant cohomology ring $H_T^\ast(M):=H^\ast(ET \times_T M)$ of $M$ with $\Z$ coefficients:
    $$
     H_T^\ast (M)\cong \Z[v_1, \ldots, v_d]/I_P = \Z(P),
    $$
    where $v_j$ is the equivariant Poincar\'e dual of the characteristic submanifold $M_j = \rho^{-1}(F_j)$ in $M$. Note that $H_T^{\ast}(M)$ is not only a ring but also an $H^{\ast}(BT) = \Z[t_1,\ldots,t_n]$-module via the map $p^\ast$, where $p\colon ET \times_T M \to BT$ is the natural projection, and $p^{\ast}$ takes $t_i$ to $\theta_i:=\lambda_{i1}v_1 + \cdots + \lambda_{id}v_d \in \Z(P)$, where $\blambda(F_i)= (\lambda_{1i} , \ldots, \lambda_{ni})^T \in \Z^n$ for $i=1,\ldots, n$.
    Since everything has vanishing odd degrees,
    $H_T^{\ast}(M)$ is a free $H^{\ast}(BT)$-module. Hence, the kernel of $\Z(P) = H_T^{\ast}(M) \rightarrow H^{\ast}(M)$ is the ideal $J_{\blambda}$ of $\Z(P)$ generated by $\theta_1, \ldots, \theta_n$. Therefore, we have
    \begin{equation*}
        H^{\ast}(M) = \Z[v_1, \ldots, v_d]/(I_P +J_{\blambda}).
    \end{equation*}
    See \cite{DJ} for more details of the previous argument.

    Let $M$ be a quasitoric manifold equivalent to a projective bundle over a $4$-dimensional quasitoric manifold $B$. Let $\pi \colon M \to B$ be the projection map. Then, the cohomology ring of $M$ has the natural $H^\ast(B)$-module structure via the induced map $\pi^\ast$ by $\pi$. Precisely, assume that its characteristic matrix is of the form \eqref{eqn:char. matrix of proj. bundle over a toric mfd 1}.
    We note that
    $$
        H^\ast(B) = \Z[x_1, \ldots, x_m] / \mathcal{I},
    $$ where $\mathcal{I}$ is an ideal generated by
    $$
        \begin{array}{cl}
        x_i(\Lambda_1^1x_1+\Lambda_1^2x_2+\cdots+\Lambda_1^mx_m) & \text{for $i=1,\ldots,m-1$,} \\
        x_j(\Lambda_2^1x_1+\Lambda_2^2x_2+\cdots+\Lambda_2^mx_m) & \text{for $j=2,\ldots,m$, and} \\
        x_kx_\ell & \text{for $k\neq\ell\pm 1$.}
        \end{array}
    $$
    Then, $H^\ast(M)$ is computed as the following :
    \begin{equation} \label{eq:cohom_quasi_bundle}
        H^\ast(M) = H^\ast(B) [x_0] \left/ x_0 \prod_{i=1}^n \left(x_0+\sum_{j=1}^m a_{i,j} x_j \right)\right.,
    \end{equation}
    where $x_j = \pi^\ast (x_j)$. Furthermore, $\pi^\ast (H^\ast(B)) \subset H^\ast(M)$ provides a natural subring $H^\ast(B)$ of $H^\ast(M)$.

\section{Combinatorial rigidity for the product of a simplex and a polygon}\label{sec:polytope rigidity}
    In this section, we claim that a quasitoric manifold whose cohomology ring is isomorphic to that of a projective bundle $P(E)$ over a $4$-dimensional quasitoric manifold has the orbit space combinatorially equivalent to $\Delta^n \times G(m+2)$. In order to do this, we introduce one important invariant of a simple polytope coming from its face ring.

    Let $P$ be a simple polytope with $d$ facets. A \emph{finite free resolution}$[R:\phi]$ of a face ring $\Q(P)$ is an exact sequence
    \begin{equation*}
        \xymatrix{ 0 \ar[r]& R^{-r} \ar[r]^{\phi}& R^{-r+1} \ar[r]^{\phi}&
        \cdots \ar[r]^{\phi} & R^{-1}\ar[r]^{\phi}& R^0 \ar[r]^{\phi}&
        \Q(P) \ar[r] & 0},
    \end{equation*}
    where $R^{-i}$ is a finite free $\Q[x_1,\ldots,x_d]$-module and each differential map $\phi$ is degree-preserving. If we take $R^{-i}$ to be the module generated by the minimal basis of $\ker\phi$, we get a \emph{minimal resolution} of $\Q(P)$. Since $\Q(P)$ is graded, so are all $R^{-i}$'s, that is, $R^{-i} = \bigoplus_{j \geq 0}R^{-i,2j}$. Let
    $$
    \beta^{-i,2j}(P) = \dim_\Q R^{-i,2j},
    $$ and we call it the $(-i,2j)$-th \emph{bigraded Betti number} of $P$.
    \begin{theorem}\label{thm:Hocster}\cite{Ho}
    Let $P$ be a simple polytope with facets $F_1, \ldots, F_d$.
    Then, we have
    $$
    \beta^{-i,2j}(P) = \sum_{\atopp{|\sigma|=j,}{\sigma \subset \{1, \ldots, d\}}} \dim \widetilde{H}_{j-i-1}(\bigcup_{i\in \sigma} F_i).
    $$
    Here, $\dim \widetilde{H}_{-1}(\emptyset)=1$ by convention.
    \end{theorem}

    Bigraded Betti numbers also satisfy the following relations (see \cite{BP} for details).
    \begin{proposition} \label{prop:betti_numbers} Let $P$ be an $n$-dimensional simple polytope with $d$ facets. Then, we have the following:
    \begin{enumerate}
      \item $\beta^{0,0}(P) = \beta^{-(d-n),2d}(P)=1$;
      \item (Poincar{\'e} duality) $\beta^{-i,2j}(P) = \beta^{-(d-n)+i, 2(d-j)}(P)$;
      \item $\beta^{-i,2j}(P_1 \times P_2) = \sum_{\atopp{i'+i''=i}{j'+j''=j}} \beta^{-i',2j'}(P_1) \beta^{-i'',2j''}(P_2)$.
    \end{enumerate}
    \end{proposition}
    \begin{example}[Bigraded Betti numbers of the product of a simplex and a polygon] \label{example:simplex_x_polygon}
        Since the union of elements in each proper subset of $\mathcal{F}(\Delta^n)$ is contractible, by Theorem~\ref{thm:Hocster}, we have that for all $i,j$, $$\beta^{-i,2j}(\Delta^n) = 0, \quad \text{except for $\beta^{0,0}(\Delta^n) = \beta^{-1,2(n+1)}(\Delta^n)=1$}.$$
        Using Theorem~\ref{thm:Hocster}, it is a good exercise to prove that
          $$
            \beta^{-i,2j}(G(m+2)) = \left\{
                                      \begin{array}{ll}
                                        1, & \hbox{if $(i,j)=(0,0), (m,m+2)$;} \\
                                        \frac{(m+2)(k-1)}{m+2-k}\binom{m}{k}, & \hbox{if $(i,j)=(k-1,k)$;} \\
                                        0, & \hbox{otherwise}
                                      \end{array}
                                    \right.
          $$ (see \cite[Corollary 3.7]{CK10}).
        Assume that $n>1$. Then, by Proposition~\ref{prop:betti_numbers}~(3), we can compute the bigraded Betti numbers of $\Delta^n \times G(m+2),$ which is an $(n+2)$-dimensional simple polytope having $m+n+3$ facets:
          $$
            \beta^{-i,2j}(\Delta^n \times G(m+2)) = \left\{
                                                      \begin{array}{ll}
                                                        1, & \hbox{\tiny if $(i,j)=(0,0), (1,n+1),$} \\
                                                        & \quad \hbox{\tiny $(m,m+2),(m+1,m+n+3)$;}\\
                                                        \frac{(m+2)(k-1)}{m+2-k}\binom{m}{k}, & \hbox{\tiny if $(i,j)=(k-1,k), (k, k+n+1)$;} \\
                                                        0, & \hbox{otherwise.}
                                                      \end{array}
                                                    \right.
          $$
    \end{example}

    \begin{definition}
    A simple polytope $P$ is (toric) \emph{cohomologically rigid} if there exists a quasitoric manifold $M$ over $P$, and whenever there exists a quasitoric manifold $N$ over a simple polytope $Q$ with a graded ring isomorphism $H^\ast(M) \cong H^\ast(N)$, $Q$ is combinatorially equivalent to $P$.
    \end{definition}
    Since Choi-Panov-Suh \cite{ch-pa-su10} showed that for two quasitoric manifolds $M$ and $N$ over $P$ and $Q$, respectively, $H^\ast(M) \cong H^\ast(N)$ implies that $\beta^{-i,2j}(P)=\beta^{-i,2j}(Q)$ for all $i$ and $j$,
    one efficient way to decide the cohomological rigidity of a simple polytope $P$ is to check the uniqueness of its bigraded Betti numbers among all simple polytopes.
    \begin{definition}
        A simple polytope $P$ is (toric) \emph{combinatorially rigid} if $Q$ is combinatorially equivalent to $P$ whenever $\beta^{-i,2j}(Q) = \beta^{-i,2j}(P)$ for all $i,j$.
    \end{definition}
    We note that if $P$ supports a quasitoric manifold and it is combinatorially rigid, then $P$ is cohomologically rigid.

    \begin{theorem}\label{thm:rigid polytope}
        A product of a simplex and a polygon is combinatorially rigid, that is, if a simple polytope $P$ satisfies $$\beta^{-i,2j}(P) = \beta^{-i,2j}(\Delta^n \times G(m+2))$$ for all $i$ and $j$, then $P$ is combinatorially equivalent to $\Delta^n \times G(m+2)$.
    \end{theorem}
    \begin{proof}
        When $m\leq 2$, $G(m+2)$ is either $\Delta^2$ or $(\Delta^1)^2$. Since the product of simplices is combinatorially rigid by \cite{ch-pa-su10}, the assertion is true.

        When $n=1$, $\Delta^1 \times G(m+2)$ is an $(m+2)$-gonal prism. From \cite{CK11}, it is also known that any prism is combinatorially rigid. Therefore, the assertion is also true for this case.

        Now, assume that $m>2$ and $n>1$.
        Since the bigraded Betti numbers determine the dimension and the number of facets of the polytope, from the assumption, $P$
        is a simple polytope of dimension $n+2$ with $m+n+3$ facets and, by Example~\ref{example:simplex_x_polygon},
        $$ \beta^{-1,2j}(P) = \left\{
                          \begin{array}{ll}
                            1, & \hbox{if $j=n+1$;} \\
                            \frac{(m+2)(m-1)}{2}, & \hbox{if $j=2$;} \\
                            0, & \hbox{otherwise.}
                          \end{array}
                        \right.
        $$ Since $\beta^{-1,2(n+1)}(P)=1$, there is a set $W$ of $n+1$ facets that satisfies $$\dim_\Q \widetilde{H}_{n-1}(\bigcup_{F \in W} F)=1.$$ Such a set is unique, and we can put $W=\{F_{m+3}, \ldots, F_{m+n+3}\} \subset \mathcal{F}(P)$.
        It follows from Alexander duality that the complement $W^c\subset \mathcal{F}(P)$, say $W^c=\{F_1,\ldots,F_{m+2}\}$, has the same homology groups as the circle $S^1$. Hence, we have
        $$\dim_\Q \widetilde{H}_1(\bigcup_{F\in{W^c}}F)=1.$$
        Moreover, since $\beta^{-j,2(n+j)}(P)=0$ for all $j>1$, by Proposition~\ref{prop:betti_numbers}~(2), we have  $\beta^{-(m-j+1),2(m-j+3)}(P)=0$, that is, any union of $m-j+3$ facets cannot be homotopy equivalent to $S^1$.
        In other words, by an appropriate re-indexing of facets of $W^c$, $F_i$ intersects with exactly two facets $F_{i-1}$ and $F_{i+1}$ for all $i=1, \ldots, m+2$, where the indices are up to modulo $m+2$.

        We note that $\beta^{-1,2j}(P)$ is the number of monomial generators of degree $j$ for $I_P$, the Stanley-Reisner ideal of $P$.
        Since
        \begin{equation*}
            \begin{split}
            \beta^{-1,4}(P) &= \sum_{|V|=2, V \subset \mathcal{F}(P)} \widetilde{H}_0(\bigcup_{F \in V} F)\\
            &\geq \sum_{|V|=2, V \subset W^c} \widetilde{H}_0(\bigcup_{F \in V} F)\\
            &=\frac{(m+2)(m-1)}{2}=\beta^{-1,4}(P),\\
            \end{split}
        \end{equation*}
        we can deduce that there are no generators for $I_P$ that are divisible by $x_k x_{m+2+j}$ for some $1\leq k \leq m+2$ and $1\leq j \leq n+1$, where $x_i$ is the indeterminate corresponding to $F_i$. Hence, $\Q(P)$ is decomposed as
        $$
        \Q(P) = \Q[x_1, \ldots, x_{m+2}]/I \otimes \Q[x_{m+3}, \ldots, x_{m+n+3}]/\langle x_{m+3} \cdots x_{m+n+3}\rangle,
        $$
        where $I$ is the ideal generated by $x_i x_j$'s with $j \not\equiv i \pm 1 (\text{mod }m+2)$.
        Hence,
        $$\Q(P) = \Q(G(m+2)) \otimes \Q(\Delta^n) = \Q(\Delta^n \times G(m+2)).$$
        Since the face ring completely determines the combinatorial type of the polytope by \cite{BG}, $P$ is combinatorially equivalent to $\Delta^n \times G(m+2)$.
    \end{proof}

    \begin{corollary}\label{cor:orbit space of P(E)}
        Let $P(E)$ be a projective bundle over a $4$-dimensional quasitoric manifold $B$. If the cohomology ring of a quasitoric manifold $M$ is isomorphic to that of $P(E)$, then the orbit space of $M$ is $\Delta^n\times G(\beta_2(B)+2)$.
    \end{corollary}

\section{Cohomology determines an equivariant bundle structure}\label{sec:cohomology determines}

    We remark that a (generalized) Bott manifold admits an iterated equivariant bundle structure.
    It was shown in~\cite{CPS12} that when the cohomology ring of a quasitoric manifold $M$ is isomorphic to that of a two-stage generalized Bott manifold $B$, the quasitoric manifold is homeomorphic to $B$. Furthermore, if the dimension of fiber is not equal to $1$, then $M$ also admits an equivariant bundle structure. On the other hand, it was shown in \cite{CS11} that if $H^\ast(M)$ is isomorphic to the cohomology ring of some Bott manifold, then $M$ should admit an iterated equivariant bundle structure to be a Bott manifold. Hence, we conclude that the cohomology ring of a quasitoric manifold should have information of an equivaraint bundle structure.

    In this section, we claim that the cohomology ring of a non-singular projective toric variety determines whether it admits a projective bundle structure over a \smoothcompact toric surface or not.

    \begin{lemma}\label{thm:quasitoric}
        Let $B$ be a $4$-dimensional quasitoric manifold and let $P(E)$ be a projective bundle over $B$ with fiber $\CP^n$. Let $M$ be a quasitoric manifold whose cohomology ring is isomorphic to $H^\ast(P(E))$. Assume
        $n\geq2$ or $\beta_2(B) \geq 3$. Then, $M$ is equivalent to a projective bundle over a $4$-dimensional quasitoric manifold.
    \end{lemma}
    \begin{proof}
        Let $\beta_2(B)=m$ and $H^\ast(B)=\Z[y_1,\ldots,y_m]/\mathcal{I}_B,$ where $\mathcal{I}_B$ is the ideal generated by $y_ky_\ell$ for $k\neq\ell\pm1$, $y_i({\Lambda^\prime}_1^1y_1+{\Lambda^\prime}_1^2y_2+\cdots+{\Lambda^\prime}_1^my_m)$ for $i=1,\ldots,m-1$, and $y_j({\Lambda^\prime}_2^1y_1+{\Lambda^\prime}_2^2y_2+\cdots+{\Lambda^\prime}_2^my_m)$ for $j=2,\ldots,m$.
        By~\eqref{eq:cohom_quasi_bundle}, we may assume that
        \begin{equation*}
            \begin{split}
                H^\ast(P(E))&=H^\ast(B)[y_0]\left/y_0\prod_{i=1}^n\left(y_0+\sum_{j=1}^ma^\prime_{ij}y_j\right)\right.\\
                &=\Z[y_0,y_1,\ldots,y_m]/\mathcal{I}',
            \end{split}
        \end{equation*}
        where $\mathcal{I}'$ is the ideal generated by a monomial
        \begin{equation}\label{eqn:monomial y0}
            y_0\prod_{i=1}^n\left(y_0+\sum_{j=1}^m a^\prime_{ij}y_j\right)
        \end{equation} and the ideal $\mathcal{I}_B$.

        By Corollary~\ref{cor:orbit space of P(E)}, the orbit space of $M$ is $\Delta^n\times G(m+2)$. Thus, we may assume that
        the characteristic matrix of $M$ is of the form
        \begin{equation}\label{eqn:char. matrix of M}
            \Lambda=\left(
            \begin{array}{ccc|cc|c|ccc}
                1 & \cdots & 0 & 0 & 0 & -1 & -a_{1,1} & \cdots & -a_{1,m} \\
                \vdots & \ddots & \vdots & \vdots & \vdots & \vdots & \vdots & \ddots & \vdots \\
                0 & \cdots & 1 & 0 & 0 & -1 & -a_{n,1} & \cdots & -a_{n,m} \\
                \hline
                0 & \cdots & 0 & 1 & 0 & -b & -\Lambda_1^1 & \cdots & -\Lambda_1^m \\
                0 & \cdots & 0 & 0 & 1 & -c & -\Lambda_2^1 & \cdots & -\Lambda_2^m
            \end{array}
            \right),
        \end{equation}
        up to equivalence\footnote{From the non-singularity condition, we obtain a column vector $(\pm1,\ldots,\pm1,-b,-c)^t$ for the $(n+3)$-th column of~\eqref{eqn:char. matrix of M}. Since the matrix that changes the row of~\eqref{eqn:char. matrix of M} is an element of $GL_{n+2}(\Z)$ and a characteristic function is determined up to sign, we can change all $+1$'s in the $(n+3)$-th column to $-1$ up to equivalence.}. Then, the cohomology ring of $M$ is
        $$H^\ast(M)=\Z[x_0,x_1,\cdots,x_m]/\mathcal{I},$$ where $\mathcal{I}$ is the ideal generated by polynomials
        $$
        \begin{array}{cl}
        x_0 \prod_{i=1}^n (x_0+\sum_{j=1}^m a_{i,j} x_j), & \\
        x_i(bx_0+\Lambda_1^1x_1+\Lambda_1^2x_2+\cdots+\Lambda_1^mx_m) & \text{for $i=1,\ldots,m-1$,} \\
        x_j(cx_0+\Lambda_2^1x_1+\Lambda_2^2x_2+\cdots+\Lambda_2^mx_m) & \text{for $j=2,\ldots,m$, and} \\
        x_kx_\ell & \text{for $k\neq\ell\pm 1$.}
        \end{array}$$ Furthermore, the matrix
        \begin{equation*}
            \Lambda_N=\begin{pmatrix}
                1 & 0 & -\Lambda_1^1 & \cdots & -\Lambda_1^m \\
                0 & 1 & -\Lambda_2^1 & \cdots & -\Lambda_2^m
            \end{pmatrix}
        \end{equation*} is a characteristic matrix on $G(m+2)$ by Lemma~\ref{block_char}.

        Note that if $b=c=0$, then the characteristic matrix of $M$ is of the form~\eqref{eqn:char. matrix of proj. bundle over a toric mfd 1}, but $M(G(m+2),\Lambda_N)$ is just a $4$-dimensional quasitoric manifold (not necessarily a \smoothcompact toric surface in general), and then, $M$ is equivalent to a quasitoric manifold that is an equivariant bundle over a $4$-dimensional quasitoric manifold with the fiber space $\CP^n$ by Lemma~\ref{lem:chr_eq_bdl}.

        From the hypothesis, there is a graded ring isomorphism $\psi \colon H^\ast(M)\to H^\ast(P(E))$. Then, $\psi$ lifts to a grading preserving isomorphism $$\bpsi\colon\Z[x_0,x_1,\ldots,x_m]\to\Z[y_0,y_1,\ldots,y_m]$$ with $\bpsi(\mathcal{I})=\mathcal{I}'$. Then, we can obtain a matrix $\mathbf{P}=[P_i^j]$, $i=0,1,\ldots,m$ and $j=0,1,\ldots,m$, such that
        \begin{equation*}
            \begin{pmatrix}
                \bpsi(x_0)\\
                \bpsi(x_1)\\
                \vdots\\
                \bpsi(x_m)
            \end{pmatrix}
            =\begin{pmatrix}
                P_0^0&P_0^1&\cdots&P_0^m\\
                P_1^0&P_1^1&\cdots&P_1^m\\
                \vdots&\vdots&\ddots&\vdots\\
                P_m^0&P_m^1&\cdots&P_m^m
            \end{pmatrix}
            \begin{pmatrix}
                y_0\\
                y_1\\
                \vdots\\
                y_m
            \end{pmatrix}
        \end{equation*}
        and $\det(\mathbf{P})=\pm 1$. We consider two cases (1) $n\geq 2$ and (2) $n=1$, separately.

        \textbf{Case 1: $n\geq 2$.}
        If $m=1$, then $P(E)$ is a two-stage generalized Bott manifold. Hence, $M$ is equivalent to a two-stage generalized Bott manifold by \cite{CPS12}. Now, assume that $m\geq 2$. For each $k=1,\ldots,m-1$, since $x_k(bx_0+\Lambda_1^1x_1+\Lambda_1^2x_2+\cdots+\Lambda_1^mx_m)$ is quadratic and $y_0\prod_{i=1}^n(y_0+\sum_{j=1}^m a^\prime_{ij}y_j)$ is not quadratic, $y_0\prod_{i=1}^n(y_0+\sum_{j=1}^m a^\prime_{ij}y_j)$ does not appear in $\bpsi(x_k(bx_0+\Lambda_1^1x_1+\Lambda_1^2x_2+\cdots+\Lambda_1^mx_m))\in\mathcal{I}'$ as a component of its linear combination. Since
        \begin{equation*}
            \begin{split}
                &\bpsi(x_k(bx_0+\Lambda_1^1x_1+\Lambda_1^2x_2+\cdots+\Lambda_1^mx_m))\\
                &\qquad=(P^0_ky_0+\cdots+P^m_ky_m)\left\{(bP^0_0+\Lambda_1^1 P^0_1 +\cdots+\Lambda_1^mP^0_m)y_0+\right.\\
                &\qquad\qquad\qquad \left.\cdots+(bP^m_0+\Lambda_1^1 P^m_1 +\cdots+\Lambda_1^mP^m_m)y_m\right\},
            \end{split}
        \end{equation*}
        we have
        \begin{equation*}
            P_k^0(bP_0^0+\cdots+\Lambda_1^mP_m^0)=0,
        \end{equation*}
        and
        \begin{equation*}
            P_k^0(bP_0^\ell +\Lambda_1^1 P^\ell_1 +\cdots+\Lambda_1^mP_m^\ell)+ P_k^\ell(bP_0^0+\Lambda_1^1 P^0_1 +\cdots+\Lambda_1^mP_m^0)=0
        \end{equation*}
        for $\ell=1,\ldots,m$. If $P_k^0\neq 0$ for some $k=1,\ldots,m-1$, then we can see that
        \begin{equation*}
            bP_0^\ell+\Lambda_1^1 P^\ell_1 +\cdots+\Lambda_1^mP_m^\ell=0
        \end{equation*}
        for all $\ell=0,1,\ldots,m$. Then, all columns of $\mathbf{P}$ are orthogonal to $(b,\Lambda_1^1,\ldots,\Lambda_1^m)$ in $\R^{m+1}$. Since $(b,\Lambda_1^1,\ldots,\Lambda_1^m)\neq\mathbf{0}$, the rank of $\mathbf{P}$ is at most $m$ which contradicts that $P$ is invertible.
        Hence, $P_k^0=0$ for $k=1,\ldots,m-1$. Similarly, since $x_m(cx_0+\Lambda_2^1x_1+\Lambda_2^2x_2+\cdots+\Lambda_2^mx_m)$ is also quadratic, we obtain $P_m^0=0$. Therefore,
        $P_k^0=0$ for all $k=1,\ldots,m$. Note that $y_0\prod_{i=1}^n(y_0+\sum_{j=1}^m a^\prime_{ij}y_j)$ is only one generator containing $y_0$ among the generators of $\mathcal{I}'$. This implies that  $b=c=0$.

        \textbf{Case 2: $n=1$ and $m\geq 3$.}
        If $F_i^2\cap F_j^2=\emptyset$, then $x_ix_j\in\mathcal{I}$ and
        we have $$\bpsi(x_ix_j)=(P_i^0y_0+P_i^1y_1+\cdots+P_i^my_m)(P_j^0y_0+P_j^1y_1+\cdots+P_j^my_m).$$ Note that if a quadratic element in $\mathcal{I}'$ contains a term divisible by $y_0$, the exponent of $y_0$ of the term must be equal to $2$. Hence,
        we can see that $\bpsi(x_ix_j)\in\mathcal{I}'$ only if either $P_i^0=P_j^0=0$ or $P_i^0P_j^0\neq 0$. Hence, by considering the elements $x_1 x_3, x_1 x_4, \ldots, x_1x_m$ in $\mathcal{I}$, the entries $P_1^0,P_3^0,\ldots,P_m^0$ are either all zero or all nonzero. Since $m\geq 3$, the four polynomials $x_1(bx_0+\Lambda_1^1x_1+\cdots+\Lambda_1^mx_m)$, $x_2(bx_0+\Lambda_1^1x_1+\cdots+\Lambda_1^mx_m)$, $x_2(cx_0+\Lambda_2^1x_1+\cdots+\Lambda_2^mx_m)$, and $x_m(cx_0+\Lambda_2^1x_1+\cdots+\Lambda_2^mx_m)$ belong to $\mathcal{I}$. Therefore, similarly, their images via $\bpsi$ belong to $\mathcal{I}'$ only if $P_1^0$, $P_2^0$, $P_m^0$, $bP_0^0+\Lambda_1^1P_1^0+\cdots+\Lambda_1^mP_m^0$, and $cP_0^0+\Lambda_2^1P_1^0+\cdots+\Lambda_2^mP_m^0$ are either all zero or all nonzero.
        One can observe that $P_1^0,\ldots,P_m^0, bP_0^0+\Lambda_1^1P_1^0+\cdots+\Lambda_1^mP_m^0$ and $cP_0^0+\Lambda_2^1P_1^0+\cdots+\Lambda_2^mP_m^0$ are either all zero or all nonzero.

        Suppose that $P_1^0,\ldots,P_m^0, bP_0^0+\Lambda_1^1P_1^0+\cdots+\Lambda_1^mP_m^0$, and $cP_0^0+\Lambda_2^1P_1^0+\cdots+\Lambda_2^mP_m^0$ are nonzero. For the pair $(i,j)$ with $j\neq i \pm 1$, we have
        \begin{equation}\label{eqn1:identity for x_ix_j}
            \begin{split}
            \bpsi(x_ix_j)&=P_i^0P_j^0y_0^2+\sum_{\ell=1}^m(P_i^\ell P_j^0+P_i^0P_j^\ell)y_0y_\ell\\
            &\qquad+(P_i^1y_1+\cdots+P_i^my_m)(P_j^1y_1+\cdots+P_j^my_m).
            \end{split}
        \end{equation}
        Since $\bpsi(x_1x_j)\in\mathcal{I}'$ for $j=3,\ldots,m$, from \eqref{eqn:monomial y0} and \eqref{eqn1:identity for x_ix_j},
        we obtain
        $$P_1^0P_j^0a_{1,\ell}^\prime=P_1^\ell P_j^0+P_1^0P_j^\ell$$ for $\ell=1,\ldots,m$. Since $P_1^0\neq0$,
        we can see that
        \begin{equation}\label{eqn:Pjl}
            P_j^\ell=\left(a_{1,\ell}'-\frac{P_1^\ell}{P_1^0}\right)P_j^0
        \end{equation} for $\ell=1,\ldots,m$.
        Since $\bpsi(x_1(bx_0+\Lambda_1^1x_1+\cdots+\Lambda_1^mx_m))\in\mathcal{I}'$, we also obtain
        \begin{equation*}
            \begin{split}
                &P_1^0(bP_0^0+\Lambda_1^1P_1^0+\cdots+\Lambda_1^mP_m^0)a_{1,\ell}^\prime\\
                &\qquad=P_1^\ell (bP_0^0+\Lambda_1^1P_1^0+\cdots+\Lambda_1^mP_m^0)+P_1^0(bP_0^\ell+\Lambda_1^1P_1^\ell+\cdots+\Lambda_1^mP_m^\ell),
            \end{split}
        \end{equation*} and hence, we can see that
        \begin{equation}\label{eqn1:P_1^b and sum of blabla}
            bP_0^\ell+\Lambda_1^1P_1^\ell+\cdots+\Lambda_1^mP_m^\ell = \left(a_{1,\ell}^\prime-\frac{P_1^\ell}{P_1^0} \right)(bP_0^0+\Lambda_1^1P_1^0+\cdots+\Lambda_1^mP_m^0).
        \end{equation} Substituting $P_j^\ell=\left(a_{1,\ell}'-\frac{P_1^\ell}{P_1^0}\right)P_j^0$ into~\eqref{eqn1:P_1^b and sum of blabla} for $j=3,\ldots,m$,
        we have
        \begin{equation}\label{eqn:P012}
            bP_0^\ell+\Lambda_1^1P_1^\ell+\Lambda_1^2P_2^\ell = \left(a_{1,\ell}'-\frac{P_1^\ell}{P_1^0}\right)(bP_0^0+\Lambda_1^1P_1^0+\Lambda_1^2P_2^0)
        \end{equation}
        for $\ell=1,\ldots,m$.

        From \eqref{eqn:Pjl} and \eqref{eqn:P012}, we have $$P_3^0(b\mathbf{P}_0+\Lambda_1^1\mathbf{P}_1+\Lambda_1^2\mathbf{P}_2) =(bP_0^0+\Lambda_1^1P_1^0+\Lambda_1^2P_2^0)\mathbf{P}_3,$$ where $\mathbf{P}_i$ is the $i$-th row vector of $\mathbf{P}$. This implies that the first four row vectors of the matrix $\mathbf{P}$
        are linearly dependent, which contradicts the claim that $\psi$ is an isomorphism. Therefore, $P_i^0=0$ for $i=1,\ldots,m$ and $P_0^0\neq 0$. At the same time, both $bP_0^0+\Lambda_1^1P_1^0+\cdots+\Lambda_1^mP_m^0$ and $cP_0^0+\Lambda_2^1P_1^0+\cdots+\Lambda_2^mP_m^0$ are zero. Therefore, $b=c=0$, and hence, By Lemma~\ref{lem:chr_eq_bdl}, $M$ is equivalent to some equivariant bundle over $B'$, where $B'$ is a $4$-dimensional quasitoric manifold whose characteristic function is
        $$
            \left(
            \begin{array}{ccccc}
              1 & 0 & -\Lambda_1^1 & \cdots & -\Lambda_1^m \\
              0 & 1 & -\Lambda_2^1 & \cdots & -\Lambda_2^m \\
            \end{array}
          \right).
        $$
        Note that, by Proposition~\ref{prop:char_ftn_projective_bundle}, there is a projective bundle over $B'$ whose characteristic matrix is of the form \eqref{eqn:char. matrix of M} with $b=c=0$. Since a quasitoric manifold is determined by its characteristic function up to equivalence, we conclude that $M$ is equivalent to a projective bundle $P(E')$ over $B'$.

    \end{proof}

    The following examples show that when $n=1$ and $\beta_2(B)\leq 2$, there exist quasitoric manifolds whose cohomology rings are isomorphic to $H^\ast(P(E))$, but which cannot admit an equivariant bundle structure.

    \begin{example}
        Let $M$ be a quasitoric manifold over $\Delta^1\times\Delta^2 = \Delta^1 \times G(3)$ with its characterstic matrix
        \begin{equation*}
            \begin{pmatrix}
                1&0&0&-1&-2\\
                0&1&0&0&-1\\
                0&0&1&0&-1
            \end{pmatrix}.
        \end{equation*}
        Then $M$ is an equivariant bundle whose base is $\CP^2$ and fiber $\CP^1$. Let $N$ be a quasitoric manifold over a cube $\Delta^1 \times \Delta^2$ with its characteristic matrix
       \begin{equation*}
            \begin{pmatrix}
                1&0&0&-1&-2\\
                0&1&0&-1&-1\\
                0&0&1&-1&-1
            \end{pmatrix}
        \end{equation*}
        which does not admit an equivariant bundle structure. However, $M$ and $N$ are homeomorphic which implies that they have the isomorphic cohomology rings (see~\cite{CPS12} for further details).
    \end{example}

    \begin{example} \label{eg:CP2sharpCP2}
        Let $M$ be a quasitoric manifold over a cube
        $(\Delta^1)^3 = \Delta^1 \times G(4)$ with its characteristic matrix
        \begin{equation*}
            \begin{pmatrix}
                1&0&0&-1&-1&-1\\
                0&1&0&0&-1&-2\\
                0&0&1&0&-1&-1
            \end{pmatrix}.
        \end{equation*}
        Then $M$ is an equivariant bundle whose base is $\CP^2\#\CP^2$ and fiber $\CP^1$.
        Now let $N$ be a quasitoric manifold over a cube $(\Delta^1)^3 $ with its characteristic matrix
        \begin{equation*}
            \begin{pmatrix}
                1&0&0&-1&-1&-1\\
                0&1&0&0&-1&-2\\
                0&0&1&-2&-1&-1
            \end{pmatrix}.
        \end{equation*}
        Then $N$ cannot have an equivariant bundle structure.

        Let us compute the cohomology rings of $M$ and $N$:
        \begin{equation*}
        H^\ast(M)=\Z[x,y,z]\left/\left\langle x(x+y+z),y(y+2z),z(y+z)\right\rangle\right.
        \end{equation*} and
        \begin{equation*}
            H^\ast(N)=\Z[X,Y,Z]\left/\left\langle X(X+Y+Z),Y(Y+2Z),Z(2X+Y+Z)\right\rangle\right..
        \end{equation*} Then the map $\phi\colon H^\ast(M)\to H^\ast(N)$ defined by $\phi(x)=-X-Y-Z$, $\phi(y)=2X+Y+2Z$ and $\phi(z)=-Z$ is a graded ring isomorphism.
    \end{example}

    \begin{remark}
         By using the smooth classification of $6$-dimensional simply connected closed manifolds due to Wall and Jupp (\cite{Wall}, \cite{Jupp}), one can see that the quasitoric manifold $N$ in each above example is diffeomorphic to $M$. This implies that $N$ admits a non-equivariant bundle structure.
    \end{remark}

    \begin{theorem}[Theorem~\ref{main1}]\label{cor:proj bld}
        A \projectivetoricmanifold is equivalent to a projective bundle over a \smoothcompact toric surface if and only if its integral cohomology ring is isomorphic to that of some projective bundle over a \smoothcompact toric surface as graded rings.
    \end{theorem}

    \begin{proof}
        Since the ``only if'' statement is obvious, it suffices to show the ``if'' statement.
        Let $M$ be a \projectivetoricmanifold and $P(E)$ a projective bundle over a \smoothcompact toric surface.
        Assume that a fiber of $P(E)$ is $\CP^n$ and $\beta_2(S) = m$. If $n=1$ and $m\leq 2$, then the orbit space of $M$ is combinatorially equivalent to that of $P(E)$ by Corollary~\ref{cor:orbit space of P(E)}, that is, either $\Delta^1 \times \Delta^2$ or $(\Delta^1)^3$. Since any \toricmanifold over a product of simplices is a generalized Bott manifold (by \cite{ch-ma-su08}), $M$ is equivalent to a two-stage generalized Bott manifold (whose base is $\CP^2$) or a three-stage Bott manifold.

        Assume that $n \geq 2$ or $m \geq 3$. By Lemma~\ref{thm:quasitoric}, $M$ is equivalent to a projective bundle over a $4$-dimensional quasitoric manifold $B$. Assume that
        $M=P(E)=P(\CC \oplus L_1 \oplus \cdots \oplus L_n)$ over $B$. We note that the zero section of $B$ as
        $P(\CC \oplus {0}) \subset P(E)$ is equivalent to $B$ itself, and is expressed as an intersection of characteristic submanifolds of $M$. Since $M$ is a non-singular complete toric variety, $B$ should be a \smoothcompact toric surface, which proves the theorem.
    \end{proof}

\section{Topology of $6$-dimensional projective bundles}\label{sec:cohomological rigidity and strong cohomological rigidity}
    \begin{lemma}\label{prop:induce base isomorphism}
        Let $M=P(E)$ and $M'=P(E')$ be quasitoric manifolds that are projective bundles over $4$-dimensional quasitoric manifolds $B$ and $B'$ with fiber $\CP^n$, respectively. Then, a graded ring isomorphism from $H^\ast(M)$ to $H^\ast(M')$ induces a graded ring isomorphism from $H^\ast(B)$ to $H^\ast(B')$ provided that $n\geq 2$ or $\beta_2(B)\geq 3$.
    \end{lemma}
    \begin{proof}
        If $n\geq2$ and $\beta_2(B)=1$, then $P(E)$ and $P(E')$ are two-stage generalized Bott manifold, and hence, any cohomology ring isomorphism from $H^\ast(P(E))$ to $H^\ast(P(E'))$ preserves the subring $H^\ast(\CP^2)$ by Lemma 6.2 in \cite{ch-ma-su10}.

        Now consider the case when $\underline{n\geq 2~\&~\beta_2(B)\geq 2}$ or $\underline{n=1~\&\beta_2(B)\geq 3}$.
        By Corollary~\ref{cor:orbit space of P(E)}, the orbit spaces of $M$ and $M'$ are combinatorially equivalent to the same polytope, a product of a simplex and a polygon, and the number of edges of the polygon is determined by $\beta_2(B)$ (or $\beta_2(B')$). Hence, $\beta_2(B) = \beta_2(B')$, say $m$.
        By~\eqref{eq:cohom_quasi_bundle}, we may assume that
        \begin{equation*}
            H^\ast(P(E))=H^\ast(B)[x_0] \left/ y_0\prod_{i=1}^n\left(x_0+\sum_{j=1}^ma_{ij}x_j\right)\right.,
        \end{equation*} and
        \begin{equation*}
            H^\ast(P(E'))=H^\ast(B')[y_0]\left/y_0\prod_{i=1}^n\left(y_0+\sum_{j=1}^ma^\prime_{ij}y_j\right)\right..
        \end{equation*}
        Let $\psi$ be a graded ring isomorphism from $H^\ast(M)$ to $H^\ast(M')$, and let $\bpsi$ be the induced grading preserving isomorphism from $\Z[x_0,x_1,\ldots,x_m]$ to $\Z[y_0,y_1,\ldots,y_m]$. Then, as in the proof of Lemma~\ref{thm:quasitoric}, $\bpsi$ induces a square matrix $\mathbf{P}$ of order $m+1$ such that $\bpsi(x_i)=\sum_{i=0}^m P_i^jy_j$ for $i=0,1,\ldots,m$, and the entries $P_i^0$ are all zero for $i=1,\ldots,m$ and that the entry $P_0^0$ is $\pm 1$ when $\underline{n\geq 2~\&~m\geq 2}$ or $\underline{n=1 ~\&~ m\geq 3}$. This implies that $\psi$ induces a graded ring isomorphism from $H^\ast(B)$ to $H^\ast(B')$.
    \end{proof}

    Note that, by \cite{OR}, a $4$-dimensional quasitoric manifold is equivariantly diffeomorphic to an equivariant connected sum of copies of $\CP^2$ and Hirzebruch surfaces. Since Hirzebruch surfaces are diffeomorphic to either $\CP^1\times\CP^1$ or $\CP^2\#\overline{\CP^2}$, any $4$-dimensional quasitoric manifold is diffeomorphic to a connected sum of copies of $\CP^2$, $\overline{\CP^2}$ and $\CP^1\times\CP^1$. Furthermore, the cohomology ring of a $4$-dimensional quasitoric manifold determines a smooth type of the manifold.

    \begin{proposition} \label{thm:quasitoric_equivalent_P(E)}
      Let $M$ be a quasitoric manifold, and let $P(E)$ be a projective bundle over a $4$-dimensional quasitoric manifold $B$ with fiber $\CP^n$. Assume that $n\geq2$ or $\beta_2(B) \geq 3$. Then, $M$ is equivalent to some projective bundle $P(E')$ over $B'$, where $B'$ is diffeomorphic to $B$.
    \end{proposition}

    \begin{proof}
        By Lemma~\ref{thm:quasitoric}, $M$ is equivalent to a projective bundle over a $4$-dimensional quasitoric manifold $B'$. By Lemma~\ref{prop:induce base isomorphism}, $H^\ast(B')$ should be isomorphic to $H^\ast(B)$ as graded rings. Therefore, $B'$ and $B$ are diffeomorphic.
    \end{proof}

    Although the cohomology ring of $\CP^1$-bundle over $\CP^2\#\CP^2$ does not determine the equivariant bundle structure as in Example~\ref{eg:CP2sharpCP2}, any cohomology ring isomorphism between projective bundles over $\CP^2\#\CP^2$ preserves the subring $H^\ast(\CP^2\#\CP^2)$ as the following.
    \begin{lemma}\label{lem:CP2sharpCP2}
        Let $B=\CP^2\#\CP^2$. If
        $P(E)$ and $P(E')$
        be projective bundles over $B$ with fiber $\CP^1$. If
        $\varphi\colon H^\ast(P(E))\to H^\ast(P(E'))$
        is an isomorphism, it preserves the subring $H^\ast(B)$.
    \end{lemma}
    \begin{proof}
        Up to equivalence, we may assume that the characteristic matrices of
        $P(E)$ and $P(E')$
         are of the form
        \begin{equation*}
            \Lambda_1:=\begin{pmatrix}
                1&0&0&-1&-a_1&-a_2\\
                0&1&0&0&-1&-2\\
                0&0&1&0&-1&-1
            \end{pmatrix} \mbox{ and }
            \Lambda_2:=\begin{pmatrix}
                1&0&0&-1&-b_1&-b_2\\
                0&1&0&0&-1&-2\\
                0&0&1&0&-1&-1
            \end{pmatrix},
        \end{equation*}respectively. Hence, the cohomology rings of
        $P(E)$ and $P(E')$
        are
        \begin{equation*}
            H^\ast(P(E))=\Z[x,y,z]\left/\left\langle x(x+2y), y(x+y), z(a_1x+a_2y+z)\right\rangle\right.
        \end{equation*} and
        \begin{equation*}
            H^\ast(P(E'))=\Z[X,Y,Z]\left/\left\langle X(X+2Y), Y(X+Y), Z(b_1X+b_2Y+Z)\right\rangle\right..
        \end{equation*}
         Since $\varphi\colon H^\ast(P(E))\to H^\ast(P(E'))$
        is an isomorphism, there exists a matrix $C=[c_{ij}]$ of size 3 such that
        \begin{equation*}
            \begin{pmatrix}
                \varphi(x)\\\varphi(y)\\\varphi(z)
            \end{pmatrix}=\begin{pmatrix}
                c_{11}&c_{12}&c_{13}\\
                c_{21}&c_{22}&c_{23}\\
                c_{31}&c_{32}&c_{33}
            \end{pmatrix}\begin{pmatrix}
                X\\Y\\Z
            \end{pmatrix}
        \end{equation*} and $\det C=\pm1$. Since $\varphi(x(x+2y))=0$ and $\varphi(y(x+y))=0$ in
        $H^\ast(P(E'))$,
        we have
        \begin{equation}\label{eq:2sharp2 compare coeff 1}
            \begin{split}
                &\varphi(x(x+2y))\\
                &\quad=(c_{11}X+c_{12}Y+c_{13}Z)((c_{11}+2c_{21})X+(c_{12}+2c_{22})Y+(c_{13}+2c_{23})Z)\\
                &\quad=c_{11}(c_{11}+2c_{21})X(X+2Y)+c_{12}(c_{12}+2c_{22})Y(X+Y)\\ &\qquad\qquad+c_{13}(c_{13}+2c_{23})Z(b_1X+b_2Y+Z)
            \end{split}
        \end{equation} and
        \begin{equation}\label{eq:2sharp2 compare coeff 2}
            \begin{split}
                &\varphi(y(x+y))\\
                &\quad=(c_{21}X+c_{22}Y+c_{23}Z)((c_{11}+c_{21})X+(c_{12}+c_{22})Y+(c_{13}+c_{23})Z)\\
                &\quad=c_{21}(c_{11}+c_{21})X(X+2Y)+c_{22}(c_{12}+c_{22})Y(X+Y)\\
                &\qquad\qquad+c_{23}(c_{13}+c_{23})Z(b_1X+b_2Y+Z).
            \end{split}
        \end{equation}
        By comparing the coefficients of $XZ$ in both sides of \eqref{eq:2sharp2 compare coeff 1}, we have
        \begin{equation}\label{eq:compare 1-2}
            c_{11}(c_{13}+2c_{23})+c_{13}(c_{11}+2c_{21})=b_1c_{13}(c_{13}+2c_{23}).
        \end{equation}
        By comparing the coefficients of $YZ$ in both sides of \eqref{eq:2sharp2 compare coeff 1}, we have
        \begin{equation}\label{eq:compare 1-3}
            c_{12}(c_{13}+2c_{23})+c_{13}(c_{12}+2c_{22})=b_2c_{13}(c_{13}+2c_{23}).
        \end{equation}

        Similarly, by comparing the coefficients of $XZ$ and $YZ$ in both sides of \eqref{eq:2sharp2 compare coeff 2}, we have the following.
        \begin{align}
                c_{21}(c_{13}+c_{23})+c_{23}(c_{11}+c_{21})&=b_1c_{23}(c_{13}+c_{23}),\mbox{ and}\label{eq:compare 2-2}\\
                c_{22}(c_{13}+c_{23})+c_{23}(c_{12}+c_{22})&=b_2c_{23}(c_{13}+c_{23}),\label{eq:compare 2-3}
        \end{align}
        respectively.

        We need to show that $c_{13}=c_{23}=0$.

        At first, suppose that $c_{13}\neq 0$ and $c_{23}=0$. From \eqref{eq:compare 2-2} and \eqref{eq:compare 2-3}, we have $c_{21}=c_{22}=0$, which contradicts that $C$ is invertible.

        Secondly, suppose that $c_{13}=0$ and $c_{23}\neq 0$. From \eqref{eq:compare 1-2} and \eqref{eq:compare 1-3}, we have $c_{11}=c_{12}=0$, which is also contradiction.

        Now, suppose that neither $c_{13}$ nor $c_{23}$ is equal to zero.

        I. Assume $c_{13}+2c_{23}=0$. From \eqref{eq:compare 1-2} and \eqref{eq:compare 1-3}, we can see that $(c_{11},c_{12},c_{13})+2(c_{21},c_{22},c_{23})=\mathbf{0}$, which is a contradiction.

        II. Assume $c_{13}+c_{23}=0$. From \eqref{eq:compare 2-2} and \eqref{eq:compare 2-3}, we obtain $(c_{11},c_{12},c_{13})+(c_{21},c_{22},c_{23})=\mathbf{0}$, which is a contradiction.

        III. Assume that neither $c_{13}+c_{23}$ nor $c_{13}+2c_{23}$ is equal to zero. Then, from \eqref{eq:compare 1-2} and \eqref{eq:compare 2-2}, we have
        \begin{equation}\label{eq:b_1}
            \frac{c_{11}(c_{13}+2c_{23})+c_{13}(c_{11}+2c_{21})}{c_{13}(c_{13}+2c_{23})}
            =\frac{c_{21}(c_{13}+c_{23})+c_{23}(c_{11}+c_{21})}{c_{23}(c_{13}+c_{23})}.
        \end{equation}
        From \eqref{eq:compare 1-3} and \eqref{eq:compare 2-3}, we have
        \begin{equation}\label{eq:b_2}
            \frac{c_{12}(c_{13}+2c_{23})+c_{13}(c_{12}+2c_{22})}{c_{13}(c_{13}+2c_{23})}
            =\frac{c_{22}(c_{13}+c_{23})+c_{23}(c_{12}+c_{22})}{c_{23}(c_{13}+c_{23})}.
        \end{equation}
        From \eqref{eq:b_1}, we have
        \begin{equation}\label{eq:b_1'}
        \begin{split}
            &c_{11}c_{23}c_{13}^2+2c_{11}c_{13}c_{23}^2+2c_{11}c_{23}^3 -2c_{13}^2c_{21}c_{23}-2c_{13}c_{21}c_{23}^2-c_{13}^3c_{21}\\
            &\qquad=(2c_{13}c_{23}+2c_{23}^2+c_{13}^2)(-c_{13}c_{22}+c_{12}c_{23})\\
            &\qquad=0.
        \end{split}
        \end{equation}
        From \eqref{eq:b_2}, we have
        \begin{equation}\label{eq:b_2'}
        \begin{split}
            &c_{12}c_{13}^2c_{23}+2c_{12}c_{13}c_{23}^2+2c_{12}c_{23}^2 -c_{13}^3c_{22}-2c_{13}^2c_{22}c_{23}-2c_{13}c_{22}c_{23}^2\\
            &\qquad=(2c_{13}c_{23}+2c_{23}^2+c_{13}^2)(-c_{13}c_{21}+c_{11}c_{23})\\
            &\qquad=0.
        \end{split}
        \end{equation}
        Since $2c_{13}c_{23}+2c_{23}^2+c_{13}^2=0$ has no integer solution, we can obtain
        $$c_{13}c_{22}=c_{12}c_{23}\mbox{ and } c_{13}c_{21}=c_{11}c_{23}$$
        which is a contradiction that $C$ is invertible.

        Therefore, $c_{13}=c_{23}=0$, and hence, $\varphi$ preserves the subring $H^\ast(B)$.
    \end{proof}
        In this paper, for a closed connected manifold $M$, we use $p(M)$ and $w(M)$ to denote the total Pontryagin class and the total Stifel-Whitney class of $M$, respectively. In order to classify projective bundles
    over a \smoothcompact toric surface, we prepare a few lemmas that indicate the invariance of characteristic classes under cohomology ring isomorphisms.

    \begin{lemma}\label{lem:4-dim'l quasitoric manifolds}
        If $B$ and $B'$ are $4$-dimensional quasitoric manifolds such that $\phi\colon H^\ast(B)\to H^\ast(B')$ is a graded ring isomorphism, then $\phi$ preserves their Pontryagin classes, that is, $\phi(p(B))=p(B')$.
    \end{lemma}
    \begin{proof}
        Note that the graded cohomology ring isomorphism $\phi$ induces the isomorphism $\varphi\colon H^2(B)\to H^2(B')$ that preserves the self-intersection form, and hence, the first Pontryagin class is preserved by $\varphi$. Therefore, one can see that $\phi(p(B))=p(B')$.
    \end{proof}

    \begin{proposition}\label{prop:Pontrjagin class preserving}\cite{ch-ma-su10}
        Let $E\to B$ and $E'\to B'$ be complex vector bundles over smooth manifolds $B$ and $B'$ with the same fiber dimension, respectively. Suppose that
        $\psi\colon H^\ast(P(E))\to H^\ast(P(E'))$ is an isomorphism such that
        $\psi(H^\ast(B))=H^\ast(B')$ and $\psi(p(B))=p(B')$; then, $\psi(p(P(E)))=p(P(E'))$.
    \end{proposition}

    \begin{lemma}\cite{ch-ma-su10}\label{lem:preserve stiefel whitney class}
        Let $M$ be a connected closed manifold of dimension $n$. Suppose that $H^\ast(M)$ is generated by $H^r(M)$ for some $r$ as a ring, and let $M'$ be another connected manifold of dimension $n$ such that $H^\ast(M';\Z/2)$ is isomorphic to $H^\ast(M;\Z/2)$ as rings. Then,
        $\psi(w(M))=w(M')$ for any ring isomorphism $\psi\colon H^\ast(M;\Z/2) \to H^\ast(M';\Z/2)$.
    \end{lemma}

    \begin{theorem}[Theorem~\ref{main2}]\label{thm:CP^1-bundles over a 4-dim toric manifold}
        Let $M$ and $M'$ be projective bundles over $4$-dimensional quasitoric manifolds $B$ and $B'$ with the fiber space $\CP^1$, respectively. If $H^\ast(M)\cong H^\ast(M')$, then $M$ and $M'$ are diffeomorphic.
    \end{theorem}
    \begin{proof}
        If a quasitoric manifold $B$ is either $\CP^2$ or a Hirzebruch surface, then $M$ is equivalent to either a two-stage generalized Bott manifold or a three-stage Bott manifold. In these two cases, they are classified by their cohomology rings up to diffeomorphism in~\cite{ch-ma-su10}.

        We only need to prove the case when $M$ is neither a two-stage generalized Bott manifold nor a three-stage Bott manifold. That is, $B$ and $B'$ are equivalent to $\CP^2\#\CP^2$, or the orbit spaces of $B$ and $B'$ are $G(m+2)$ with $m\geq 3$. By Lemma~\ref{prop:induce base isomorphism} and Lemma \ref{lem:CP2sharpCP2}, any graded ring isomorphism $\psi$ from $H^\ast(M)$ to $H^\ast(M')$ induces an isomorphism between $H^\ast(B)$ and $H^\ast(B')$. Then, by Lemma~\ref{lem:4-dim'l quasitoric manifolds}, $\psi$ satisfies $\psi(p(B))=p(B')$. Hence, by Proposition~\ref{prop:Pontrjagin class preserving}, $\psi$ satisfies $\psi(p(M))=p(M')$. By Lemma~\ref{lem:preserve stiefel whitney class}, $\psi$ also preserves Stiefel-Whitney classes. Hence, the isomorphism $\psi$ preserves their Stiefel-Whitney classes and Pontryagin classes. Therefore, the $6$-dimensional \toricmanifolds $M$ and $M'$ are diffeomorphic by~\cite{Jupp} and \cite{Wall}.
    \end{proof}

    \section{Topology of higher dimensional projective bundles} \label{section:higher_projective_bdl}

    \begin{lemma}\label{lem:relationshp betweewn total chern and bundle}{\cite{Pet}}
        Let $X$ be a finite CW-complex such that $H^{odd}(X)=0$ and $H^\ast(X)$ has no torsion. Then, complex $n$-dimensional vector bundles over $X$ with $2n\geq \dim X$ are isomorphic if and only if their total Chern classes are the same.
    \end{lemma}

    \begin{proposition}\label{prop:SCR}
        Assume that two Whitney sums of complex line bundles over a quasitoric manifold $B$ are isomorphic if and only if their total Chern classes are the same. Let $E=\CC\oplus L_1 \oplus\cdots\oplus L_n$ be the Whitney sum of complex line bundles over $B$. If
        $\varphi$ is an $H^\ast(B)$-algebra automorphism of $H^\ast(P(E))$,
        then $\varphi$ is induced by a self-diffeomorphism on $P(E)$.
    \end{proposition}

    \begin{proof}
        Before the proof, the authors would like to inform that the the proof is quite similar to the proof of Proposition 4.3 in \cite{CP12}. We give the brief proof in order to avoid repeating the same argument.

        Let $x$ be the negative of the first Chern class of the tautological line bundle over $P(E)$. Set $c_1(L_i)=\alpha_i\in H^2(B)$, and write $\varphi(x)=\epsilon x+\omega$, where $\epsilon=\pm q$ and $\omega\in H^2(B)$. Denotes $\gamma^\alpha$ be the line bundle over $B$ whose first Chern class is $\alpha\in H^2(B)$. Remark that $L_i=\gamma^{\alpha_i}$ for all $i$.

        (I) First, assume that
        $\varphi(x)=x+\omega$. Then we have
        \begin{equation}\label{eq:s=1}
            (x+\omega)(x+\omega+\alpha_1)\cdots(x+\omega+\alpha_n)=x(x+\alpha_1)\cdots(x+\alpha_n)
        \end{equation}
        in $H^\ast(P(E))$. Comparing the coefficients of $x^n$ in both sides of \eqref{eq:s=1}, we can see that
        $\omega=0$.
        Hence, $\varphi$ is the identity.

        (II) Now assume that
        $\varphi(x)=-x+\omega$. Then we have
        \begin{equation}\label{eq:s=-1}
            (-x+\omega)(-x+\omega+\alpha_1)\cdots(-x+\omega+\alpha_n)=(-1)^{n+1}x(x+\alpha_1)\cdots(x+\alpha_n)
        \end{equation}
        in $H^\ast(P(E))$. Comparing the coefficients of $x^n$ in both sides of \eqref{eq:s=-1}, we can see that
        $\omega=-\frac{2}{n+1}\sum_{i=1}^n\alpha_i$.
        By substituting $x=1$ into \eqref{eq:s=-1}, we obtain
        \begin{equation}\label{eq:compare chern classes}
            (1-\omega)(1-\omega-\alpha_1)\cdots(1-\omega-\alpha_n)=(1+\alpha_1)\cdots(1+\alpha_n)
        \end{equation}
        in $H^*(B)$. By hypothesis, we can see that
        \begin{equation*}
            \begin{split}
                &\CC\oplus\gamma^{\alpha_1}\oplus\cdots\oplus\gamma^{\alpha_n}\\
                &=\gamma^{-\omega} \oplus \gamma^{-\omega-\alpha_1}\oplus\cdots\oplus\gamma^{-\omega-\alpha_n}\\
                &=\gamma^{-\omega}\otimes (\gamma^{-\alpha_1}\oplus\cdots\oplus\gamma^{-\alpha_n})
            \end{split}
        \end{equation*}

        Note that since $E$ possesses a Hermitian metric, its dual bundle $E^\ast=\Hom(E,\CC)$ is canonically isomorphic to the conjugate bundle $\CC\oplus\gamma^{-\alpha_1}\oplus\cdots\oplus\gamma^{-\alpha_n}$. Furthermore, the bundle map $$b\colon P(\CC\oplus\gamma^{\alpha_1}\oplus\cdots\oplus\gamma^{\alpha_n})\to P(\CC\oplus\gamma^{-\alpha_1}\oplus\cdots\oplus\gamma^{-\alpha_n})$$ induces an isomorphism $b^\ast\colon H^\ast(P(E^\ast))\to H^\ast(P(E))$ such that $b^\ast(y)=-x$, where $y$ is the negative of the first Chern class of the tautological line bundle over $P(E^\ast)$.

        Since $\gamma^{-\omega}$ is a line bundle, there exists a bundle isomorphism $$c\colon P(\CC\oplus\gamma^{-\alpha_1}\oplus\cdots\oplus\gamma^{-\alpha_n}) \to P((\CC\oplus\gamma^{-\alpha_1}\oplus\cdots\oplus\gamma^{-\alpha_n}) \otimes\gamma^{-\omega})$$
        which preserves the complex structures on each fiber. One can show that $c^\ast(x)=y+\omega$. Therefore,
        $$b^\ast(c^\ast(x))=-x+\omega.$$

        See \cite{CP12} for the explicit constructions of two maps $b$ and $c$, and the computations of $b^\ast(y)$ and $c^\ast(x)$.
    \end{proof}

    We recall again the smooth classification of $4$-dimensional quasitoric manifolds from \cite{OR}. Any $4$-dimensional quasitoric manifold is a connected sum of several copies of $\CP^2$, $\overline{\CP^2}$ and $\CP^1\times\CP^1$.
    Let $\mathcal{B}$ be the set of $4$-dimensional quasitoric manifolds which cannot be expressed as $\CP^2\#n\overline{\CP^2}$ or $n\CP^2\#\overline{\CP^2}$ for $n>9$.

    \begin{theorem}\cite{Wall64} \label{thm:Wall64}
        Let $B \in \mathcal{B}$. Then any cohomology ring automorphism of $B$ is realizable by a self-diffeomorphism of $B$.
    \end{theorem}

    We remark that, by \cite{FM}, $\mathcal{B}$ is the maximal subset of the set of $4$-quasitoric manifolds, satisfying such property.

    \begin{theorem}[Theorem~\ref{main3}]\label{thm:CP^n bundles}
        Let $B,B' \in \mathcal{B}$, and let $M$ and $M'$ be projective bundles over $B$ and $B'$, respectively. Then, any graded ring isomorphism from $H^\ast(M)$ to $H^\ast(M')$ is induced by a diffeomorphism from $M'$ to $M$.
    \end{theorem}

    \begin{proof}
       Let $n$ be the dimension of the fiber of $M$, and $m = \beta_2 (B)$.

       If $m = 1$, then both $M$ and $M'$ are $2$-stage generalized Bott manifolds. By \cite{CP12}, any cohomology
       ring isomorphism between them is realizable by
       a diffeomorphism.

       Now, assume that $m \geq 2$. In the special case where $n = 1$ and $m =2$, then $M$ is either a $3$-stage Bott manifold or a projective bundle over $\CP^2 \sharp \CP^2$ since $\Delta^1 \times \Delta^1 = G(4)$ supports either a
       $2$-stage Bott manifold or $\CP^2 \sharp \CP^2$. Since any quasitoric manifold whose cohomology ring is isomorphic to that of a Bott manifold is indeed equivalent to some Bott manifold by \cite{CS11},
       if $M$ is equivalent to a $3$-stage Bott manifold, then $M'$ is also equivalent to a $3$-stage Bott manifold. By \cite{Choi13}, any cohomology ring isomorphism between $3$-stage Bott manifolds $M$ and $M'$ is realizable by a diffeomorphism.

       If $m\geq2$ but $M$ is not equivalent to a $3$-stage Bott manifold, i.e., $B=\CP^2\#\CP^2$ or $n\geq 2$ or $m\geq 3$, then
       by Lemma~\ref{lem:CP2sharpCP2} and Lemma~\ref{prop:induce base isomorphism}, the restriction to $H^\ast(B)$ of a graded ring isomorphism $\varphi \colon H^\ast(M) \to H^\ast(M')$ is an isomorphism from $H^\ast(B)$ to $H^\ast(B')$, say $\psi$.
       Then, by Proposition~\ref{thm:quasitoric_equivalent_P(E)}, there is a diffeomorphism $g \colon B \to B'$ and $g^\ast \circ \psi$ is an automorphism of $H^\ast(B)$. Since $B \in \mathcal{B}$, by Theorem~\ref{thm:Wall64}, there is a self-automorphism $h$ of $B$ such that $h^\ast = g^\ast \circ \psi$. Consider a diffeomorphism $f = g \circ h^{-1} \colon B \to B'$. Then, $\psi$ is induced from $f^{-1}$.

       Take the pull-back bundle $f^\ast M'$ over $B$ via the map $f$. Then, it induces a diffeomorphism $\tilde{f} \colon f^\ast M' \to M'$. We note that $\tilde{f}^\ast \circ \varphi$ is an isomorphism whose restriction to $H^\ast(B)$ is an identity map. Therefore, by Lemma~\ref{lem:relationshp betweewn total chern and bundle} and Proposition~\ref{prop:SCR}, $\tilde{f}^\ast \circ \varphi$ is induced from a diffeomorphism $f' \colon f^\ast M' \to M$. Hence, $\varphi = (\tilde{f}^\ast)^{-1} \circ (f'^{-1})^\ast = (f' \circ \tilde{f}^{-1})^\ast$, and $f' \circ \tilde{f}^{-1} \colon M' \to M$ is a diffeomorphism. See the following
       commutative diagram.
        $$
        \xymatrix@1{
         H^\ast(M)\ar[dr]_{\varphi}\ar[drr]^{(f')^\ast}&&\\
         H^\ast(B)\ar[u]\ar[dr]_{\psi}\ar@{=}[drr] \ar[ddr]_{h^\ast} & H^\ast(M') \ar[r]_{\tilde{f}^\ast} & H^\ast(f^\ast M') \\
          & H^\ast(B')\ar[u] \ar[r]_{f^\ast} \ar[d]_{g^\ast} & H^\ast(B)\ar[u]\\
          & H^\ast(B) \ar[ur]_{(h^{-1})^\ast}&
        }
        $$
    \end{proof}


\bigskip
\begin{thebibliography}{amsplain}

\bibitem{BH}
A. Borel and F. Hirzebruch, \emph{Characteristic classes and homogeneous spaces I}, Amer. J. Math. 80 (1959), 485 -- 538.

\bibitem{BG}
W. Bruns and J. Gubeladze, \emph{Combinatorial invariance of Stanley-Reisner rings}, Georgian Math. J. 3 (1996), 315--318.

\bibitem{BP}
V.~M. Buchstaber and T.~E. Panov, \emph{Torus actions and their
  applications in topology and combinatorics}, University Lecture Series,
  vol.~24, American Mathematical Society, Providence, RI, 2002.

\bibitem{Choi13}
S. Choi, \emph{Classification of Bott manifolds up to dimension eight}, see arXiv:1112.2321

\bibitem{CK10}
S. Choi and J.~S. Kim, \emph{A combinatorial proof of a formula for Betti numbers of a stacked polytope}, Electron. J. Combin. 17 (2010), no.~1, \#R9.

\bibitem{CK11}
S. Choi and J.~S. Kim, \emph{Combinatorial Rigidity of 3-dimensional simplicial polytopes},  Int. Math. Res. Not. IMRN. 2011 (2011), no.~8, 1935--1951.

\bibitem{ch-ma-su11}
S. Choi, M. Masuda and D.~Y. Suh, \emph{Rigidity problems in toric topology, a survey}, Proc. Steklov Inst. Math. 275 (2011), 177--190.

\bibitem{ch-ma-su08}
S. Choi, M. Masuda and D.~Y. Suh, \emph{Quasitoric manifolds over
  a product of simplices}, Osaka J. Math. 47 (2010), no.~1, 109--129.

\bibitem{ch-ma-su10}
S. Choi, M. Masuda and D.~Y. Suh, \emph{Topological classification of generalized {B}ott towers}, Trans. Amer. Math. Soc. 362 (2010), no.~2, 1097 -- 1112.

\bibitem{ch-pa-su10}
S. Choi, T.~E. Panov, and D.~Y. Suh, \emph{Toric cohomological
  rigidity of simple convex polytopes},  J. London Math. Soc. 82 (2010), no.~2, 343--360.

\bibitem{CP12}
S. Choi and S. Park, \emph{Strong cohomological rigidity of toric varieties}, in preparation.

\bibitem{CPS12}
S. Choi, S. Park, and D.~Y. Suh, \emph{Topological classification of quasitoric manifolds with second {B}etti number $2$}, Pacific J. Math. 256 (2012), no.~1, 19 -- 49.

\bibitem{CS11}
S. Choi and D.~Y. Suh, \emph{Properties of Bott manifolds and cohomological rigidity}, Algebr. Geom. Topol. 11 (2011), no.~2, 1053--1076.

\bibitem{DJ}
M.~W. Davis and T. Januszkiewicz, \emph{Convex polytopes, {C}oxeter
  orbifolds and torus actions}, Duke Math. J. 62 (1991), no.~2,
  417--451.

\bibitem{Do}
N. {\`E}. Dobrinskaya, \emph{The classification problem for quasitoric manifolds over a given polytope}, Funktsional. Anal. i Prilozhen. 35 (2001), no.~2, 3--11, 95.

\bibitem{FY}
S. Fischli and D. Yavin, \emph{Which $4$-manifolds are toric varieites?}, Math. Z. 215 (1994), 179--185.

\bibitem{FM}
R. Friedman and J. W. Morgan, \emph{On the diffeomorphism types of certain algebraic surfaces. I},
J. Differential Geom. 27 (1988), no.~2, 297--369.

\bibitem{F}
W. Fulton, \emph{Introduction to toric varieties}, Princeton University Press, 1993.

\bibitem{Ho}
M. Hochster, \emph{Cohen-Macaulay rings, combinatorics, and simplicial complexes}, in: Ring Theory II (Proc. Second Oklahoma Conference), B. R. McDonald and R. Morris, eds.,
Dekker, New York, 1977, pp. 171--223.

\bibitem{Jupp}
P. E. Jupp, \emph{Classification of certain $6$-manifolds}, Math. Proc. Camb. Phil. Soc., 73 (1973), 293--300.

\bibitem{OR}
P. Orlik and F. Raymond, \emph{Actions of the torus on 4-manifolds, I}, Trans. Amer.
Math. Soc. 152 (1970), 531--559.

\bibitem{Pet}
F. P. Peterson, \emph{Some remarks on Chern classes}, Ann. of Math. 69 (1959), no.~2, 414--
420.


\bibitem{Wall}
C. T. C. Wall, \emph{Classification problem in differential topology. V:On certain 6-manifolds}, Invent. Math. 1 (1996), 355--374.

\bibitem{Wall64}
C. T. C. Wall, \emph{Diffeomorphisms of 4-manifolds}, J. London Math. Soc. 39 (1964), 131--140.

\end{thebibliography}
\end{document}